\documentclass[11pt,a4paper]{amsart}
\usepackage{}
\usepackage[utf8]{inputenc}

\usepackage{amssymb}
\usepackage[all]{xy}
\usepackage{indentfirst}
\usepackage{amsmath,amscd}
\usepackage{tikz,tikz-cd}

\allowdisplaybreaks

\newtheorem{theorem}{Theorem}[section]
\newtheorem{proposition}[theorem]{Proposition}
\newtheorem{lemma}[theorem]{Lemma}
\newtheorem{corollary}[theorem]{Corollary}

\theoremstyle{definition}
\newtheorem{remark}[theorem]{Remark}
\newtheorem{definition}[theorem]{Definition}
\newtheorem{example}[theorem]{Example}
\numberwithin{equation}{section}

\begin{document}

\author{Xueru Wu,  Yao Ma, Liangyun Chen$^*$}
\address{School of Mathematics and Statistics, Northeast Normal University,
Changchun, 130024, P.R.China}
\email{wuxr884@nenu.edu.cn.}
\address{School of Mathematics and Statistics, Northeast Normal University,
Changchun, 130024, P.R.China}
\email{may703@nenu.edu.cn.}
\address{School of Mathematics and Statistics, Northeast Normal University,
Changchun, 130024, P.R.China}
\email{chenly640@nenu.edu.cn.}

\thanks{*Corresponding author.}


\thanks{Supported by  NSF of Jilin Province (No. YDZJ202201ZYTS589), NNSF of China (No. 12071405) and  CSC of China (202106625001).}

\title{Deformations of relative Rota-Baxter operators on Leibniz Triple Systems}

\begin{abstract}
In this paper, we introduce the cohomology theory of relative Rota-Baxter operators on Leibniz triple systems. We use the cohomological approach to study linear and formal deformations of relative Rota-Baxter operators. In particular, formal deformations and extendibility of order $n$ deformations of a relative Rota-Baxter operators are also characterized in terms of the cohomology theory. We also consider the relationship between cohomology of relative Rota-Baxter operators on Leibniz algebras and associated Leibniz triple systems.
\end{abstract}
\maketitle

\section{Introduction}

The notion of Leibniz triple systems was introduced in {\rm\cite{Bremner}} for the first time. Leibniz triple systems were obtained by using Kolesnikov-Pozhideav algorithm {\rm\cite{Co4}} to Lie triple systems. This algorithm takes the defining identities for a variety of algebras and produced the defining identities for the corresponding variety of dialgebras, such as, one can obtain associative dialgebras from associative algebras. From the relationship between Lie triple systems and Leibniz triple systems, it is natural to extend some properties of Lie triple systems to Leibniz triple systems. So far, people have given its Levi's theorem {\rm\cite{Ma}}, multiplicative basis {\rm\cite{Di}}, centroid {\rm\cite{Cao}} for Leibniz triple systems, we have defined the cohomology {\rm\cite{WCM}} and Nijenhuis operators {\rm\cite{WMC1}} on Leibniz triple systems.

In 1960, the author introduced the notion of Rota-Baxter operators on associative algebras \cite{Baxter}. A Rota-Baxter operator on a Lie algebra was given in the 1980s as the operator form of the classical Yang-Baxter equation. Then in \cite{Kupershmidt}, the author showed more general concept of $\mathcal{O}$-operators (relative Rota-Baxter operators) on Lie algebras, which can be traced back to \cite{Bordemann}. In \cite{AMM}, the authors applied Rota-Baxter operators on super-type algebras, which build relationship between associative superalgebras, Lie superalgebras, $L$-dendriform superalgebras and per-Lie superalgebras. In \cite{BGLW}, the authors introduced the notion of  Rota-Baxter operators on a $3$-Lie algebra, and the cohomology theory of relative Rota-Baxter operators on a $3$-Lie algebra was given in \cite{THS}. The deformations of relative Rota-Baxter operators on Leibniz algebras were researched in \cite{TSZ}. $\mathcal{O}$-operators on Lie triple systems, see \cite{CHMM}. In \cite{ZQ}, the authors introduced the cohomologies and deformations of relative Rota-Baxter operators on Lie-Yamaguti algebras. The relationship between relative Rota-Baxter operators and Nijenhuis operators on Hom-Lie algebras, see \cite{DS}.

In \cite{Gerstenhaber}, the author introduced the deformation of algebraic structures when he studied associative algebras. Then, the deformation theory was extended to Lie algebras \cite{NR}, Leibniz algebras \cite{Balavoine}, Hom-type algebras \cite{MS}, $3$-Lie algebras \cite{Figueroa} and Lie triple systems \cite{YBB}. Recently, deformations of morphisms and Rota-Baxter operators were deeply studied \cite{ABM,Das,FZ,TBGS}. In this paper, we study deformations of relative Rota-Baxter operators on Leibniz triple systems. We do this with the aid of cohomology theory. First, we define the cohomology of relative Rota-Baxter operators on Leibniz triple systems. Then the formal deformations of relative Rota-Baxter operators on Leibniz triple system  are given and we show that the infinitesimal of a formal deformation is a $1$-cocycle. The extendibility of order $n$ deformations is also discussed.

The paper is organized as follows. In Section 2, we recall representations and cohomologies of Leibniz triple systems. In Section 3, we study  the relative Rota-Baxter operators on Lie triple systems with respect to representations, and  we show that the  relationship between the relative Rota-Baxter operator and the Nijenhuis operators  on a Leibniz triple system. In Section 4, we define the cohomology of relative Rota-Baxter operators on Leibniz triple systems based on the cohomology of the Leibniz triple system on the representation space. In Section 5, we investigate linear deformations of the relative Rota-Baxter operator on Leibniz triple systems and we prove that the infinitesimal of two equivalent deformations are the same cohomology class. We also consider the extensibility of the order $n$ deformations of a relative Rota-Baxter operator. In Section 6, we study some connections between cohomology of relative Rota-Baxter operators on Leibniz algebras and associated Leibniz triple systems.

In this paper, all Leibniz triple systems are defined over a fixed but arbitrary field $\mathbb{F}$.

\section{Preliminaries}

In this section, we recall some basic definitions of Leibniz triple systems.
\begin{definition}{\rm\cite{Bremner}}\label{basicdef}
A Leibniz triple system is a vector space $\mathfrak{L}$ endowed with a trilinear operation $ \{ \cdot, \cdot, \cdot \}: \mathfrak{L}\times \mathfrak{L}\times \mathfrak{L}\longrightarrow \mathfrak{L}$ satisfying
\begin{align}
& \{ a, b, \{ c, d, e \} \} = \{ \{ a, b, c \}, d, e \} - \{ \{ a, b, d \}, c, e \} - \{ \{ a, b, e \}, c, d \} + \{ \{ a, b, e \}, d, c \},\label{222.1}\\
& \{ a, \{ b, c, d \}, e \} = \{ \{ a, b, c \}, d, e \} - \{ \{ a, c, b \}, d, e \} - \{ \{ a, d, b \}, c, e \} + \{ \{ a, d, c \}, b, e \},\label{222.2}
\end{align}
for all $a, b, c, d, e \in \mathfrak{L}$.
\end{definition}
One can obtain a Leibniz triple system from a Lie triple system with the same ternary product. A Leibniz algebra $L$ with product $[\cdot, \cdot]$ becomes a Leibniz triple system when $\{x, y, z\}:= [[x, y], z],$ for all $x, y, z\in L.$ More examples see {\rm\cite{Bremner}}. Denote by End$(\mathfrak{L})$ the set consisting of all linear maps on a Leibniz triple system $(\mathfrak{L}, \{\cdot, \cdot, \cdot\})$.

A morphism $f:(\mathfrak{L},\{\cdot,\cdot,\cdot\})\rightarrow (\mathfrak{L}',\{\cdot,\cdot,\cdot\}')$ of Leibniz triple systems is a linear map satisfying
\begin{align*}
f(\{x,y,z\})=\{f(x),f(y),f(z)\}',~~\forall ~x,y,z\in \mathfrak{L}.
\end{align*}
If $f$  is a bijective morphism, then we call it an isomorphism.

\begin{definition}{\rm\cite{Ma}}
Let $(\mathfrak{L},\{\cdot,\cdot,\cdot\})$ be a Leibniz triple system and $V$ a vector space. $V$ is called an $\mathfrak{L}$-module, if $\mathfrak{L}\dot{+}V$ is a Leibniz triple system such that $(1)$ $\mathfrak{L}$ is a subsystem, $(2)$ $\{a,b,c\}\in V$ if any one of $a,b,c \in V;$ $(3)$ $\{a,b,c\}=0$ if any two of $a,b,c \in V.$
\end{definition}

\begin{definition}{\rm\cite{Ma}}\label{def282,3}
Let $(\mathfrak{L},\{\cdot,\cdot,\cdot\})$ be a Leibniz triple system and $V$ a vector space. Suppose $l,m, r : \mathfrak{L} \times \mathfrak{L} \longrightarrow$ End$(V)$ are bilinear maps such that
\begin{align*}
l(a,\{b, c, d\})&=l(\{a, b, c\},d) - l(\{a, c, b\},d) - l(\{a,d,b\},c) + l(\{a,d,c\},b),\\
m(a, d)l(b, c)&=m(\{a, b, c\}, d) - m(\{a, c, b\}, d) - r(c, d)m(a, b) + r(b, d)m(a, c),\\
m(a, d)m(b, c)&=r(c, d)l(a, b) - r(c, d)m(a, b) - m(\{a, c, b\}, d) + r(b, d)l(a, c),\\
m(a, d)r(b, c)&=r(c, d)m(a, b) - r(c, d)l(a, b) - r(b, d)l(a, c) + m(\{a, c, b\}, d),\\
r(\{a, b, c\}, d)&=r(c, d)r(a, b) - r(c, d)r(b, a) - r(b, d)r(c, a) + r(a, d)r(c, b),\\
l(a, b)l(c, d)&=l(\{a, b, c\}, d) - l(\{a, b, d\}, c) - r(c, d)l(a, b) + r(d, c)l(a, b),\\
l(a, b)m(c, d)&=m(\{a, b, c\}, d) - r(c, d)l(a, b) - l(\{a, b, d\}, c) + m(\{a, b, d\}, c),\\
l(a, b)r(c, d)&=r(c, d)l(a, b) - m(\{a, b, c\}, d) - m(\{a, b, d\}, c) + l(\{a, b, d\}, c),\\
m(a, \{b, c, d\})&=r(c, d)m(a, b) - r(b, d)m(a, c) - r(b, c)m(a, d) + r(c, b)m(a, d),\\
r(a, \{b, c, d\})&=r(c, d)r(a, b) - r(b, d)r(a, c) - r(b, c)r(a, d) + r(c, b)r(a, d),
\end{align*}
for all $a, b, c, d \in \mathfrak{L}.$ Then $(r, m, l)$ is called a representation of $\mathfrak{L}$ on $V$. In particular, if $V=\mathfrak{L}$ and $l(a,b)c=\{a,b,c\}$, $m(a,b)c=\{a,c,b\}$, $r(a,b)c=\{c,a,b\}$, then $(r, m, l)$ is called  an adjoint representation of $\mathfrak{L}$ on itself.
\end{definition}

In \cite{Ma}, the authors use a semi-direct product to describe the representation, i.e., $(r,m,l)$ is a representation of a Leibniz triple system $(\mathfrak{L},\{\cdot,\cdot,\cdot\})$ on $V$ if and only if $(\mathfrak{L}\oplus V,\{\cdot,\cdot,\cdot\}_{\mathfrak{L}\oplus V})$ is a Leibniz triple system, where $\{\cdot,\cdot,\cdot\}_{\mathfrak{L}\oplus V}$ is defined by
\begin{align}\label{2882.3}
\{x+u,y+v,z+w\}_{\mathfrak{L}\oplus V}=\{x,y,z\}+l(x,y)w+m(x,z)v+r(y,z)u,
\end{align}
for any $x,y,z\in \mathfrak{L},$ $u,v,w\in V$.

\begin{definition}{\rm\cite{WCM}}
Let $V$ be an $\mathfrak{L}$-module. A $(2n+1)$-linear map $f:\underbrace{\mathfrak{L}\otimes\cdots\otimes \mathfrak{L}}_{2n+1~{\rm times}}\rightarrow V$ is called a $(2n+1)$-cochain of $\mathfrak{L}$ on $V.$ Denote by $C^{2n+1}(\mathfrak{L},V)$ the set of all $(2n+1)$-cochains, for $n\geq 0.$
\end{definition}

\begin{definition}{\rm\cite{WCM}}
Let $(\mathfrak{L},\{\cdot,\cdot,\cdot\})$ be a Leibniz triple system and $(r,m,l)$ a representation of $\mathfrak{L}$ on $V.$ For $n=1,~2,$ the coboundary operator $\delta^{2n-1}:C^{2n-1}(\mathfrak{L},V)\rightarrow C^{2n+1}(\mathfrak{L},V)$ is defined as follows.\\
$(1)$ A $1$-coboundary operator of $\mathfrak{L}$ on $V$ is defined by
\begin{align*}
\delta^1:C^1(\mathfrak{L},V) &\rightarrow C^3(\mathfrak{L},V)\\
                            f&\mapsto \delta^1f,
\end{align*}
for $f\in C^1(\mathfrak{L},V)$ and
\begin{align*}
\delta^1f(x_1,x_2,x_3)=r(x_2,x_3)f(x_1)+m(x_1,x_3)f(x_2)+l(x_1,x_2)f(x_3)-f(\{x_1,x_2,x_3\}).
\end{align*}
$(2)$ A $3$-coboundary operator of $\mathfrak{L}$ on $V$ consists a pair of maps $(\delta^3_1,\delta^3_2),$ where
\begin{align*}
\delta^3_i:C^3(\mathfrak{L},V)&\rightarrow C^5(\mathfrak{L},V)\\
                            f&\mapsto \delta^3_if,
\end{align*}
for $f\in C^3(\mathfrak{L},V)$ and
\begin{align*}
&\delta^3_1f(x_1,x_2,x_3,x_4,x_5)\\
=&~f(x_1,x_2,\{x_3,x_4,x_5\})-f(\{x_1,x_2,x_3\},x_4,x_5)+f(\{x_1,x_2,x_4\},x_3,x_5)+f(\{x_1,x_2,x_5\},x_3,x_4)\\
&-f(\{x_1,x_2,x_5\},x_4,x_3)+l(x_1,x_2)f(x_3,x_4,x_5)-r(x_4,x_5)f(x_1,x_2,x_3)+r(x_3,x_5)f(x_1,x_2,x_4)\\
&+r(x_3,x_4)f(x_1,x_2,x_5)-r(x_4,x_3)f(x_1,x_2,x_5),
\end{align*}
\begin{align*}
&\delta^3_2f(x_1,x_2,x_3,x_4,x_5)\\
=&~f(x_1,\{x_2,x_3,x_4\},x_5)-f(\{x_1,x_2,x_3\},x_4,x_5)+f(\{x_1,x_3,x_2\},x_4,x_5)+f(\{x_1,x_4,x_2\},x_3,x_5)\\
&-f(\{x_1,x_4,x_3\},x_2,x_5)+m(x_1,x_5)f(x_2,x_3,x_4)-r(x_4,x_5)f(x_1,x_2,x_3)+r(x_4,x_5)f(x_1,x_3,x_2)\\
&+r(x_3,x_5)f(x_1,x_4,x_2)-r(x_2,x_5)f(x_1,x_4,x_3).
\end{align*}
\end{definition}

Let $(\mathfrak{L},\{\cdot,\cdot,\cdot\})$ be a Leibniz triple system and $V$ an $\mathfrak{L}$-module. The set
\begin{align*}
Z^1(\mathfrak{L},V)=\{f\in C^1(\mathfrak{L},V)~|~\delta^1f=0\}
\end{align*}
is called the space of $1$-cocycles of $\mathfrak{L}$ on $V$.

The set
\begin{align*}
Z^3(\mathfrak{L},V)=\{f\in C^3(\mathfrak{L},V)~|~\delta^3_1f=\delta^3_2f=0\}
\end{align*}
is called the space of $3$-cocycles of $\mathfrak{L}$ on $V$. The set
\begin{align*}
B^3(\mathfrak{L},V)=\{\delta^1f~|~f\in C^1(\mathfrak{L},V)\}
\end{align*}
is called the space of $3$-coboundaries of $\mathfrak{L}$ on $V$.

Then the $1$-cohomology space and $3$-cohomology space of $(\mathfrak{L},\{\cdot,\cdot,\cdot\})$ are defined as
\begin{align*}
H^1(\mathfrak{L},V):=&Z^1(\mathfrak{L},V).\\
H^3(\mathfrak{L},V):=&Z^3(\mathfrak{L},V)/B^3(\mathfrak{L},V).
\end{align*}

\section{Relative Rota-Baxter operators on Leibniz triple systems}

In this section, we give the definition of Rota-Baxter operators and relative Rota-Baxter operators on Leibniz triple systems. Moreover, we give some characterisation of relative Rota-Baxter operators in terms of Nijenhuis operators and graphs.

\begin{definition}
Let $(\mathfrak{L},\{\cdot,\cdot,\cdot\})$ be a Leibniz triple system. A linear map $R: \mathfrak{L}\rightarrow \mathfrak{L}$ is called a Rota-Baxter operator if it satisfies
\begin{align*}
\{R(x),R(y),R(z)\}=R\Big(\{R(x),R(y),z\}+\{R(x),y,R(z)\}+\{x,R(y),R(z)\}\Big),
\end{align*}
for all $x,y,z\in \mathfrak{L}$.
\end{definition}

\begin{example}
Consider a $2$-dimensional Leibniz triple system $(\mathfrak{L},\{\cdot,\cdot,\cdot\})$ given with respect to a basis $\{x,y\}$ by $$\{x,y,y\}=\{y,y,y\}=x.$$ Then
$
R=
\left(
\begin{matrix}
a_{11}&a_{12}\\
a_{21}&a_{22}
\end{matrix}
\right)
$ is a Rota-Baxter operator on $(\mathfrak{L},\{\cdot,\cdot,\cdot\})$ if and only if
$$\{Ra,Rb,Rc\}=R\Big(\{Ra,Rb,c\}+\{Ra,b,Rc\}+\{a,Rb,Rc\}\Big),$$ for all $a,b,c\in \mathfrak{L}$. We have
\begin{align*}
\{Rx,Ry,Ry\}=&\{a_{11}x+a_{21}y,a_{12}x+a_{22}y,a_{12}x+a_{22}y\}\\
            =&(a_{11}a_{22}^2+a_{21}a_{22}^2)x,
\end{align*}
and
\begin{align*}
&R\Big(\{Rx,Ry,y\}+\{Rx,y,Ry\}+\{x,Ry,Ry\}\Big)\\
=&R\Big(\{a_{11}x+a_{21}y,a_{12}x+a_{22}y,y\}+\{a_{11}x+a_{21}y,y,a_{12}x+a_{22}y\}+\{x,a_{12}x+a_{22}y,a_{12}x+a_{22}y\}\Big)\\
=&R\Big(a_{11}a_{22}x+a_{21}a_{22}x+a_{11}a_{22}x+a_{21}a_{22}x+a_{22}^2x\Big)\\
=&(2a_{11}a_{22}+2a_{21}a_{22}+a_{22}^2)(a_{11}x+a_{21}y).
\end{align*}
Thus, we have $(2a_{11}a_{22}+2a_{21}a_{22}+a_{22}^2)a_{11}=a_{11}a_{22}^2+a_{21}a_{22}^2$ and $a_{21}(2a_{11}a_{22}+2a_{21}a_{22}+a_{22}^2)=0$.

Similarly, we also have
$\{Ry,Ry,Ry\}=R\Big(\{Ry,Ry,y\}+\{Ry,y,Ry\}+\{y,Ry,Ry\}\Big),$ i.e.,
$a_{11}(2a_{12}a_{22}+3a_{22}^2)=a_{12}a_{22}^2+a_{22}^3$ and $a_{21}(2a_{12}a_{22}+3a_{22}^2)=0$.

Hence, we have $
R=
\left(
\begin{matrix}
a_{11}&a_{12}\\
0&0
\end{matrix}
\right)
$ or
$
R=
\left(
\begin{matrix}
0&a_{12}\\
0 &-a_{12}
\end{matrix}
\right)
$.
\end{example}

The definition of relative Rota-Baxter operators (also called $\mathcal{O}$-operators or Kupershmidt operators) is a generalization of Rota-Baxter operators in any representation.
\begin{definition}\label{def283.2}
Let $V$ be an $\mathfrak{L}$-module and $(r,m,l)$ a representation of the Leibniz triple system $(\mathfrak{L},\{\cdot,\cdot,\cdot\})$ on $V.$ A linear operator $T:V\rightarrow \mathfrak{L}$ is called a relative Rota-Baxter operators on $(\mathfrak{L},\{\cdot,\cdot,\cdot\})$ with respect to $(r,m,l),$ if $T$ satisfies:
\begin{align}
\{Tu,Tv,Tw\}=T\Big(l(Tu,Tv)w+m(Tu,Tw)v+r(Tv,Tw)u\Big),\label{282.3}
\end{align}
for any $u,v,w\in V.$
\end{definition}

\begin{remark}
If $V=\mathfrak{L}$, then $T$ is a Rota-Baxter operator on $(\mathfrak{L},\{\cdot,\cdot,\cdot\})$ with respect to the adjoint representation.
\end{remark}

We will show that a relative Rota-Baxter operator can be lifted up to a Rota-Baxter operator, in the following proposition.
\begin{proposition}
Let $(r,m,l)$ be a representation of $(\mathfrak{L},\{\cdot,\cdot,\cdot\})$ on $V$ and $T: V\rightarrow \mathfrak{L}$ a linear map. Define $\widehat{T}\in {\rm End}(\mathfrak{L}\oplus V)$ by $\widehat{T}(x+u)=Tu$. Then $T$ is a relative Rota-Baxter operator if and only if $\widehat{T}$ is a Rota-Baxter operator on $\mathfrak{L}\oplus V$.
\end{proposition}
\begin{proof}
For any $x,y,z\in \mathfrak{L}$ and $u,v,w\in V$, we have
\begin{align*}
&\{\widehat{T}(x+u),\widehat{T}(y+v),\widehat{T}(z+w)\}_{\mathfrak{L}\oplus V}-\widehat{T}\Big(\{\widehat{T}(x+u),\widehat{T}(y+v),z+w\}_{\mathfrak{L}\oplus V}\\
&+\{\widehat{T}(x+u),y+v,\widehat{T}(z+w)\}_{\mathfrak{L}\oplus V}
+\{x+u,\widehat{T}(y+v),\widehat{T}(z+w)\}_{\mathfrak{L}\oplus V}\Big)\\
=&\{Tu,Tv,Tw\}_{\mathfrak{L}\oplus V}-\widehat{T}\Big(\{Tu,Tv,z+w\}_{\mathfrak{L}\oplus V}
+\{Tu,y+v,Tw\}_{\mathfrak{L}\oplus V}+\{x+u,Tv,Tw\}_{\mathfrak{L}\oplus V}\Big)\\
=&\{Tu,Tv,Tw\}-T\big(l(Tu,Tv)w+m(Tu,Tw)v+r(Tv,Tw)u\big).
\end{align*}
Then $\widehat{T}$ is a Rota-Baxter operator on $\mathfrak{L}\oplus V$ if and only if $T$ is a relative Rota-Baxter operator on $\mathfrak{L}$.
\end{proof}

Now, we are going to characterize relative Rota-Baxter operators in terms of graphs.

\begin{proposition}
A linear map $T:V\rightarrow \mathfrak{L}$ is a relative Rota-Baxter operator if and only if the graph
\begin{align*}
Gr(T)=\{Tu+u~|~u\in V\}
\end{align*}
is a subalgebra of the semi-direct product $\mathfrak{L}\oplus V$.
\end{proposition}
\begin{proof}
For all $u,v,w\in \mathfrak{L}$, we have
\begin{align*}
\{Tu+u,Tv+v,Tw+w\}_{\mathfrak{L}\oplus V}=\{Tu,Tv,Tw\}+l(Tu,Tv)w+m(Tu,Tw)v+r(Tv,Tw)u,
\end{align*}
which implies that $Gr(T)$ is a subalgebra of $\mathfrak{L}\oplus V$, if and only if $T$ satisfies
\begin{align*}
\{Tu,Tv,Tw\}=T\Big(l(Tu,Tv)w+m(Tu,Tw)v+r(Tv,Tw)u\Big),
\end{align*}
hence, $T$ is a relative Rota-Baxter operator.
\end{proof}

In the sequel, we give the relationship between relative Rota-Baxter operators and Nijenhuis operators. Recall from \cite{WMC1} that a Nijenhuis operator on a Leibniz triple system $(\mathfrak{L},\{\cdot,\cdot,\cdot\})$ is a linear map $N:\mathfrak{L}\rightarrow \mathfrak{L}$, satisfying
\begin{align*}
\{Nx,Ny,Nz\}=&N\{Nx,Ny,z\}+N\{x,Ny,Nz\}+N\{Nx,y,Nz\}-N^2\{Nx,y,z\}\notag\\
&-N^2\{x,Ny,z\}-N^2\{x,y,Nz\}+N^3\{x,y,z\},
\end{align*}
for all $x,y,z\in \mathfrak{L}.$

\begin{proposition}
Let $(\mathfrak{L},\{\cdot,\cdot,\cdot\})$ be a Leibniz triple system and $(r,m,l)$ a representation of $\mathfrak{L}$ on $V$. A linear map $T:V\rightarrow \mathfrak{L}$ is a relative Rota-Baxter operator if and only if
\begin{align*}
\widetilde{T}=
\left(
\begin{matrix}
{\rm id}&T\\
0&0
\end{matrix}
\right):\mathfrak{L}\oplus V\rightarrow \mathfrak{L}\oplus V
\end{align*}
is a Nijenhuis operator on $(\mathfrak{L}\oplus V, \{\cdot,\cdot,\cdot\}_{\mathfrak{L}\oplus V})$.
\end{proposition}
\begin{proof}
For all $x,y,z\in \mathfrak{L}$, $u,v,w\in V$, on the one hand, we have
\begin{align*}
&\{\widetilde{T}(x+u),\widetilde{T}(y+v),\widetilde{T}(z+w)\}_{\mathfrak{L}\oplus V}\\
=&\{x,y,z\}+\{Tu,y,z\}+\{x,Tv,z\}+\{x,y,Tw\}+\{x,Tv,Tw\}+\{Tu,y,Tw\}\\
&+\{Tu,Tv,z\}+\{Tu,Tv,Tw\}.
\end{align*}

On the other hand, since $\widetilde{T}^2=\widetilde{T}$, we have
\begin{align*}
&\widetilde{T}\Big(\{\widetilde{T}(x+u),\widetilde{T}(y+v),z+w\}_{\mathfrak{L}\oplus V}+\{\widetilde{T}(x+u),y+v,\widetilde{T}(z+w)\}_{\mathfrak{L}\oplus V}\\
&+\{x+u,\widetilde{T}(y+v),\widetilde{T}(z+w)\}_{\mathfrak{L}\oplus V}\Big)-\widetilde{T}^2\Big(\{\widetilde{T}(x+u),y+v,z+w\}_{\mathfrak{L}\oplus V}\\
&+\{x+u,y+v,\widetilde{T}(z+w)\}_{\mathfrak{L}\oplus V}
+\{x+u,\widetilde{T}(y+v),z+w\}_{\mathfrak{L}\oplus V}\Big)\\
&+\widetilde{T}^3\Big(\{x+u,y+v,z+w\}_{\mathfrak{L}\oplus V}\Big)\\
=&\{x,y,z\}+\{Tu,y,z\}+\{x,Tv,z\}+\{x,y,Tw\}+\{x,Tv,Tw\}
+\{Tu,y,Tw\}\\
&+\{Tu,Tv,z\}+T\Big(l(Tu,Tv)w+m(Tu,Tw)v+r(Tv,Tw)u\Big),
\end{align*}
which implies that $\widetilde{T}$ is a Nijenhuis operator on the Leibniz triple system $(\mathfrak{L}\oplus V, \{\cdot,\cdot,\cdot\}_{\mathfrak{L}\oplus V})$ if and only if Eq. (\ref{282.3}) holds.
\end{proof}

Now, we use the relative Rota-Baxter operator on Leibniz triple systems to characterize the Nijenhuis operator.

\begin{definition}
Two relative Rota-Baxter operators $T_1, T_2 :V \rightarrow\mathfrak{L}$ are said to be compatible if for all $\lambda, \eta\in \mathbb{\mathbb{F}}$, the sum $\lambda T_1+\eta T_2 : V\rightarrow \mathfrak{L}$ is a relative Rota-Baxter operator.
\end{definition}

The following result is straightforward.

\begin{lemma}
Two relative Rota-Baxter operators $T_1$ and $T_2$ are compatible if and only if the following conditions hold
\begin{gather}
\begin{aligned}\label{283.22}
&\{T_iu,T_iv,T_jw\}+\{T_iu,T_jv,T_iw\}+\{T_ju,T_iv,T_iw\}\\
=&T_i\Big(l(T_iu,T_jv)w+m(T_iu,T_jw)v+r(T_iv,T_jw)u\Big)\\
&+T_i\Big(l(T_ju,T_iv)w+m(T_ju,T_iw)v+r(T_jv,T_iw)u\Big)\\
&+T_j\Big(l(T_iu,T_iv)w+m(T_iu,T_iw)v+r(T_iv,T_iw)u\Big),
\end{aligned}
\end{gather}
where $i,j\in \{1,2\}$ and $i\neq j$, for all $u,v,w\in V$.
\end{lemma}
\begin{proof}
We only need to show that $\lambda T_1+\eta T_2$ is a relative Rota-Baxter operator, for all $\lambda, \eta\in \mathbb{\mathbb{F}}$. On the one hand, we have,
\begin{align*}
&\{(\lambda T_1+\eta T_2)(u),(\lambda T_1+\eta T_2)(v),(\lambda T_1+\eta T_2)(w)\}\\
=&\lambda^3\{T_1u,T_1v,T_1w\}+\lambda^2\eta\{T_1u,T_1v,T_2w\}+\lambda^2\eta\{T_1u,T_2v,T_1w\}+\lambda\eta^2\{T_1u,T_2v,T_2w\}\\
&+\lambda^2\eta\{T_2u,T_1v,T_1w\}+\lambda\eta^2\{T_2u,T_1v,T_2w\}+\lambda\eta^2\{T_2u,T_2v,T_1w\}+\eta^3\{T_2u,T_2v,T_2w\},
\end{align*}
on the other hand,
\begin{align*}
&(\lambda T_1+\eta T_2)\Big(l((\lambda T_1+\eta T_2)u,(\lambda T_1+\eta T_2)v)w+m((\lambda T_1+\eta T_2)u,(\lambda T_1+\eta T_2)w)v\\
&+r((\lambda T_1+\eta T_2)v,(\lambda T_1+\eta T_2)w)u\Big)\\
=&\lambda^3 T_1\Big(l(T_1u,T_1v)w+m(T_1u,T_1w)v+r(T_1v,T_1w)u\Big)\\
&+\lambda^2\eta T_1\Big(l(T_1u,T_2v)w+m(T_1u,T_2w)v+r(T_1v,T_2w)u\Big)\\
&+\lambda^2\eta T_1\Big(l(T_2u,T_1v)w+m(T_2u,T_1w)v+r(T_2v,T_1w)u\Big)\\
&+\lambda\eta^2 T_1\Big(l(T_2u,T_2v)w+m(T_2u,T_2w)v+r(T_2v,T_2w)u\Big)\\
&+\lambda^2\eta T_2\Big(l(T_1u,T_1v)w+m(T_1u,T_1w)v+r(T_1v,T_1w)u\Big)\\
&+\lambda\eta^2 T_2\Big(l(T_1u,T_2v)w+m(T_1u,T_2w)v+r(T_1v,T_2w)u\Big)\\
&+\lambda\eta^2 T_2\Big(l(T_2u,T_1v)w+m(T_2u,T_1w)v+r(T_2v,T_1w)u\Big)\\
&+\eta^3 T_2\Big(l(T_2u,T_2v)w+m(T_2u,T_2w)v+r(T_2v,T_2w)u\Big).
\end{align*}
So $\lambda T_1+\eta T_2$ is a relative Rota-Baxter operator if and only if Eq. (\ref{283.22}) holds.
\end{proof}

\begin{proposition}
Let $T_1$ and $T_2$ be two compatible relative Rota-Baxter operators in which $T_2$ $($resp. $T_1$$)$ is invertible. Then $N=T_1T_2^{-1}$ $($resp. $N=T_2T_1^{-1}$$)$ is a Nijenhuis operator on $(\mathfrak{L}, \{\cdot,\cdot,\cdot\})$.
\end{proposition}
\begin{proof}
We only prove the case in which $T_2$ is invertible. The other case is similar. For any $x,y,z\in \mathfrak{L}$, there exists $u,v,w\in V$, such that $T_2(u)=x, T_2(v)=y, T_2(w)=z,$ (as $T_2$ is invertible), we have
\begin{align*}
\{N(x),N(y),N(z)\}=&\{NT_2(u),NT_2(v),NT_2(w)\}=\{T_1(u),T_1(v),T_1(w)\}\\
                  =&T_1\Big(l(T_1u,T_1v)w+m(T_1u,T_1w)v+r(T_1v,T_1w)u\Big)\\
                  =&NT_2\Big(l(T_1u,T_1v)w+m(T_1u,T_1w)v+r(T_1v,T_1w)u\Big).
\end{align*}

Then,
\begin{align*}
&\{N(x),N(y),z\}+\{N(x),y,N(z)\}+\{x,N(y),N(z)\}-N\{N(x),y,z\}-N\{x,N(y),z\}\\
&-N\{x,y,N(z)\}+N^2\{x,y,z\}\\
=&\{NT_2(u),NT_2(v),T_2(w)\}+\{NT_2(u),T_2(v),NT_2(w)\}+\{T_2(u),NT_2(v),NT_2(w)\}\\
&-N\{NT_2(u),T_2(v),T_2(w)\}-N\{T_2(u),NT_2(v),T_2(w)\}-N\{T_2(u),T_2(v),NT_2(w)\}\\
&+N^2\{x,y,z\}\\
=&\{T_1(u),T_1(v),T_2(w)\}+\{T_1(u),T_2(v),T_1(w)\}+\{T_2(u),T_1(v),T_1(w)\}\\
&-N\{T_1(u),T_2(v),T_2(w)\}-N\{T_2(u),T_1(v),T_2(w)\}-N\{T_2(u),T_2(v),T_1(w)\}\\
&+N^2\{T_2(u),T_2(v),T_2(w)\}\\
=&T_1\Big(l(T_1u,T_2v)w+m(T_1u,T_2w)v+r(T_1v,T_2w)u\Big)\\
&+T_1\Big(l(T_2u,T_1v)w+m(T_2u,T_1w)v+r(T_2v,T_1w)u\Big)\\
&+T_2\Big(l(T_1u,T_1v)w+m(T_1u,T_1w)v+r(T_1v,T_1w)u\Big)\\
&-N\Big(T_2\big(l(T_2u,T_1v)w+m(T_2u,T_1w)v+r(T_2v,T_1w)u\big)\\
&~~~~~~~~~~~~~~~~~~~~~+T_2\big(l(T_1u,T_2v)w+m(T_1u,T_2w)v+r(T_1v,T_2w)u\big)\\
&~~~~~~~~~~~~~~~~~~~~~+T_1\big(l(T_2u,T_2v)w+m(T_2u,T_2w)v+r(T_2v,T_2w)u\big)\Big)\\
&+N^2T_2(l(T_2u,T_2v)w+m(T_2u,T_2w)v+r(T_2v,T_2w)u\big)\\
=&\{T_1(u),T_1(v),T_2(w)\}+\{T_1(u),T_2(v),T_1(w)\}+\{T_2(u),T_1(v),T_1(w)\}\\
&-N\{T_1(u),T_2(v),T_2(w)\}-N\{T_2(u),T_1(v),T_2(w)\}-N\{T_2(u),T_2(v),T_1(w)\}\\
&+N^2\{T_2(u),T_2(v),T_2(w)\}\\
\overset{(\ref{283.22})}{=}&T_1\Big(l(T_1u,T_2v)w+m(T_1u,T_2w)v+r(T_1v,T_2w)u\Big)\\
&+T_1\Big(l(T_2u,T_1v)w+m(T_2u,T_1w)v+r(T_2v,T_1w)u\Big)\\
&+T_2\Big(l(T_1u,T_1v)w+m(T_1u,T_1w)v+r(T_1v,T_1w)u\Big)\\
&-T_1\big(l(T_2u,T_1v)w+m(T_2u,T_1w)v+r(T_2v,T_1w)u\big)\\
&~~~~~~~~~~~~~~~~~~~~~-T_1\big(l(T_1u,T_2v)w+m(T_1u,T_2w)v+r(T_1v,T_2w)u\big)\\
&~~~~~~~~~~~~~~~~~~~~~-NT_1\big(l(T_2u,T_2v)w+m(T_2u,T_2w)v+r(T_2v,T_2w)u\big)\\
&+NT_1(l(T_2u,T_2v)w+m(T_2u,T_2w)v+r(T_2v,T_2w)u\big)\\
=&T_2\Big(l(T_1u,T_1v)w+m(T_1u,T_1w)v+r(T_1v,T_1w)u\Big).
\end{align*}
We have shown that $N$ is a Nijenhuis operator.
\end{proof}

\section{Cohomology of relative Rota-Baxter operators on Leibniz triple systems}
In this section, we define a cohomology of a relative Rota-Baxter operator $T$ on a Leibniz triple system  $(\mathfrak{L}, \{\cdot,\cdot,\cdot\})$ with coefficients in a suitable representation on $(\mathfrak{L}, \{\cdot,\cdot,\cdot\}).$ Later, we will use this cohomology to study deformation of $T$.

\begin{lemma}\label{lem282.7}
Let $T$ be a relative Rota-Baxter operator on Leibniz triple system $(\mathfrak{L},\{\cdot,\cdot,\cdot\})$ with respect to a representation $(r,m,l).$ Define
\begin{align*}
\{u,v,w\}_T=l(Tu,Tv)w+m(Tu,Tw)v+r(Tv,Tw)u,
\end{align*}
for any $u,v,w\in V.$ Then $(V, \{\cdot,\cdot,\cdot\}_T)$ is a Leibniz triple system. Moreover, $T$ is a homomorphism of the Leibniz triple system.
\end{lemma}
\begin{proof}
We have to prove that $(V, \{\cdot,\cdot,\cdot\}_T)$ satisfies Eqs. (\ref{222.1}) and (\ref{222.2}), here we only check Eq. (\ref{222.1}) as an example, one can check Eq. (\ref{222.2}) in the same way. For $(V, \{\cdot,\cdot,\cdot\}_T)$ satisfies Eq. (\ref{222.1}), i.e.,
\begin{gather}
\begin{aligned}
\underbrace{\{x,y,\{u,v,w\}_T\}_T}_{I_1}=&\underbrace{\{\{x,y,u\}_T,v,w\}_T}_{I_2}-\underbrace{\{\{x,y,v\}_T,u,w\}_T}_{I_3}
-\underbrace{\{\{x,y,w\}_T,u,v\}_T}_{I_4}\\
&+\underbrace{\{\{x,y,w\}_T,v,u\}_T,}_{I_5}\label{282.6}
\end{aligned}
\end{gather}
for all $u,v,w,x,y\in V,$ we have
\begin{align*}
I_1=&\{x,y,l(Tu,Tv)w+m(Tu,Tw)v+r(Tv,Tw)u\}_T\\
   =&\{x,y,l(Tu,Tv)w\}_T+\{x,y,m(Tu,Tw)v\}_T+\{x,y,r(Tv,Tw)u\}_T\\
   =&l(Tx,Ty)l(Tu,Tv)w+m(Tx,Tl(Tu,Tv)w)y+r(Ty,Tl(Tu,Tv)w)x\\
   &+l(Tx,Ty)m(Tu,Tw)v+m(Tx,Tm(Tu,Tw)v)y+r(Ty,Tm(Tu,Tw)v)x\\
   &+l(Tx,Ty)r(Tv,Tw)u+m(Tx,Tr(Tv,Tw)u)y+r(Ty,Tr(Tv,Tw)u)x,
\end{align*}
similarly, one has
\begin{align*}
I_2=&l(Tl(Tx,Ty)u,Tv)w+m(Tl(Tx,Ty)u,Tw)v+r(Tv,Tw)l(Tx,Ty)u\\
    &+l(Tm(Tx,Tu)y,Tv)w+m(Tm(Tx,Tu)y,Tw)v+r(Tv,Tw)m(Tx,Tu)y\\
    &+l(Tr(Ty,Tu)x,Tv)w+m(Tr(Ty,Tu)x,Tw)v+r(Tv,Tw)r(Ty,Tu)x,
\end{align*}
\begin{align*}
I_3=&l(Tl(Tx,Ty)v,Tu)w+m(Tl(Tx,Ty)v,Tw)u+r(Tu,Tw)l(Tx,Ty)v\\
    &+l(Tm(Tx,Tv)y,Tu)w+m(Tm(Tx,Tv)y,Tw)u+r(Tu,Tw)m(Tx,Tv)y\\
    &+l(Tr(Ty,Tv)x,Tu)w+m(Tr(Ty,Tv)x,Tw)u+r(Tu,Tw)r(Ty,Tv)x,
\end{align*}
\begin{align*}
I_4=&l(Tl(Tx,Ty)w,Tu)v+m(Tl(Tx,Ty)w,Tv)u+r(Tu,Tv)l(Tx,Ty)w\\
    &+l(Tm(Tx,Tw)y,Tu)v+m(Tm(Tx,Tw)y,Tv)u+r(Tu,Tv)m(Tx,Tw)y\\
    &+l(Tr(Ty,Tw)x,Tu)v+m(Tr(Ty,Tw)x,Tv)u+r(Tu,Tv)r(Ty,Tw)x,
\end{align*}
\begin{align*}
I_5=&l(Tl(Tx,Ty)w,Tv)u+m(Tl(Tx,Ty)w,Tu)v+r(Tv,Tu)l(Tx,Ty)w\\
    &+l(Tm(Tx,Tw)y,Tv)u+m(Tm(Tx,Tw)y,Tu)v+r(Tv,Tu)m(Tx,Tw)y\\
    &+l(Tr(Ty,Tw)x,Tv)u+m(Tr(Ty,Tw)x,Tu)v+r(Tv,Tu)r(Ty,Tw)x,
\end{align*}
using Definition \ref{def282,3} and Eq. (\ref{282.3}) for $I_1,$ $I_2,$ $I_3,$ $I_4$ and $I_5,$ it is easy to prove Eq. (\ref{282.6}) is true. Hence, $(V, \{\cdot,\cdot,\cdot\}_T)$ satisfies Eq. (\ref{222.1}).

Similarly, one can check that $(V, \{\cdot,\cdot,\cdot\}_T)$ satisfies Eq. (\ref{222.2}), that is to say $(V, \{\cdot,\cdot,\cdot\}_T)$ is a Leibniz triple system.
\end{proof}

\begin{theorem}
Let $V$ be an $\mathfrak{L}$-module and $T$ a relative Rota-Baxter operator on a Leibniz triple system $(\mathfrak{L}, \{\cdot,\cdot,\cdot\})$ with respect to $(r,m,l).$ Define $r_T, m_T, l_T:V\times V\rightarrow {\rm End}(\mathfrak{L})$ by
\begin{align*}
l_T(u,v)x=&\{Tu,Tv,x\}-T\Big(r(Tv,x)u+m(Tu,x)v\Big),\\
m_T(u,v)x=&\{Tu,x,Tv\}-T\Big(l(Tu,x)v+r(x,Tv)u\Big),\\
r_T(u,v)x=&\{x,Tu,Tv\}-T\Big(l(x,Tu)v+m(x,Tv)u\Big),
\end{align*}
for any $u,v\in V,$ $x\in \mathfrak{L}.$ Then $(r_T,m_T,l_T)$ is a representation of the Leibniz triple system $(V, \{\cdot,\cdot,\cdot\}_T).$
\end{theorem}
\begin{proof}
One can show it directly by a tedious computation. Inspired by \cite{THS},  we take a different approach using Nijenhuis operators on Leibniz triple system $(\mathfrak{L}, \{\cdot,\cdot,\cdot\}).$ Let $N:\mathfrak{L}\rightarrow\mathfrak{L}$ be a Nijenhuis operator on a Leibniz triple system $(\mathfrak{L}, \{\cdot,\cdot,\cdot\})$. Then $(\mathfrak{L}, \{\cdot,\cdot,\cdot\}_N)$ is a Leibniz triple system, where $\{\cdot,\cdot,\cdot\}_N$ is given by
\begin{gather}
\begin{aligned}
\{x,y,z\}_N=&\{Nx,Ny,z\}+\{Nx,y,Nz\}+\{x,Ny,Nz\}-N\{Nx,y,z\}-N\{x,Ny,z\}\\
             &-N\{x,y,Nz\}+N^2\{x,y,z\},\label{282.7}
\end{aligned}
\end{gather}
for all $x,y,z\in \mathfrak{L}.$ Let $T$ be a relative Rota-Baxter operator on a Leibniz triple system $(\mathfrak{L}, \{\cdot,\cdot,\cdot\})$ with respect to a representation $(r,m,l).$ We define $\widehat{T}:\mathfrak{L}\oplus V\rightarrow \mathfrak{L}\oplus V$ by
\begin{align*}
\widehat{T}(x+u)=Tu,
\end{align*}
for any $x\in \mathfrak{L},$ $u\in V.$ Then $\widehat{T}$ is a Nijenhuis operator on the semidirect product Leibniz triple system $\mathfrak{L}\ltimes V$ and $\widehat{T}\circ\widehat{T}=0.$ Then by Eq. (\ref{282.7}), there is a Leibniz triple system structure $\{\cdot,\cdot,\cdot\}_{\widehat{T}}$ on the vector space $\mathfrak{L}\oplus V$ by
\begin{align*}
&\{x+u,y+v,z+w\}_{\widehat{T}}\\
=&\{\widehat{T}(x+u),\widehat{T}(y+v),z+w\}+\{\widehat{T}(x+u),y+v,\widehat{T}(z+w)\}+\{x+u,\widehat{T}(y+v),\widehat{T}(z+w)\}\\
                           &-\widehat{T}\{\widehat{T}(x+u),y+v,z+w\}
                           -\widehat{T}\{x+u,\widehat{T}(y+v),z+w\}-\widehat{T}\{x+u,y+v,\widehat{T}(z+w)\}\\
                           &+\widehat{T}^2\{x+u,y+v,z+w\}\\
                           =&\{Tu,Tv,z+w\}+\{Tu,y+v,Tw\}+\{x+u,Tv,Tw\}-\widehat{T}\{Tu,y+v,z+w\}\\
                           &-\widehat{T}\{x+u,Tv,z+w\}-\widehat{T}\{x+u,y+v,Tw\}\\
                           =&\{Tu,Tv,z\}+l(Tu,Tv)w+\{Tu,y,Tw\}+m(Tu,Tw)v+\{x,Tv,Tw\}+r(Tv,Tw)u\\
                           &-Tl(Tu,y)w-Tm(Tu,z)v-Tl(x,Tv)w-Tr(Tv,z)u-Tr(y,Tw)u-Tm(x,Tw)v\\
                           =&\{u,v,w\}_T+l_T(u,v)z+r_T(v,w)x+m_T(u,w)y,
\end{align*}
for all $x,y,z\in \mathfrak{L}$ and $u,v,w\in V.$ Since a semidirect product of Leibniz triple systems is equivalent to representation of a Leibniz triple system, we deduce that $(r_T,m_T,l_T)$ is a representation of the Leibniz triple system $(V,\{\cdot,\cdot,\cdot\}_T)$ on $(\mathfrak{L},\{\cdot,\cdot,\cdot\}).$
\end{proof}

Let $\partial^{2n-1}:C^{2n-1}_T(V,\mathfrak{L})\rightarrow C^{2n+1}_T(V,\mathfrak{L})$ be the corresponding coboundary operator of the Leibniz triple system $(V,\{\cdot,\cdot,\cdot\}_T)$ with coefficient in the representation $(r_T,m_T,l_T)$ and its defined as follows

A $1$-coboundary operator of $V$ on $\mathfrak{L}$ is defined by
\begin{align*}
\partial^1:C^{1}_T(V,\mathfrak{L})&\rightarrow C^{3}_T(V,\mathfrak{L})\\
                          f&\mapsto \partial^1f,
\end{align*}
for $f\in C^{1}_T(V,\mathfrak{L})$ and
\begin{align*}
\partial^1f(v_1,v_2,v_3)=&r_T(v_2,v_3)f(v_1)+m_T(v_1,v_3)f(v_2)+l_T(v_1,v_2)f(v_3)-f(\{v_1,v_2,v_3\}_T)\\
                        =&\{f(v_1),Tv_2,Tv_3\}-T\Big(l(f(v_1),Tv_2)v_3+m(f(v_1),Tv_3)v_2\Big)\\
                         &+\{Tv_1,f(v_2),Tv_3\}-T\Big(l(Tv_1,f(v_2))v_3+r(f(v_2),Tv_3)v_1\Big)\\
                         &+\{Tv_1,Tv_2,f(v_3)\}-T\Big(m(Tv_1,f(v_3))v_2+r(Tv_2,f(v_3))v_1\Big)\\
                         &-f(l(Tv_1,Tv_2)v_3+m(Tv_1,Tv_3)v_2+r(Tv_2,Tv_3)v_1).
\end{align*}

A $3$-coboundary operator of $V$ on $\mathfrak{L}$ is a pair of maps $(\partial_1^3,\partial_1^3),$ where
\begin{align*}
\partial_i^3:C^{3}_T(V,\mathfrak{L})&\rightarrow C^{5}_T(V,\mathfrak{L})(i=1,2)\\
                          g&\mapsto \partial_i^3g,
\end{align*}
for $g\in C^{3}_T(V,\mathfrak{L})$ and
\begin{align*}
&\partial_1^3g(v_1,v_2,v_3,v_4,v_5)\\
=&g(v_1,v_2,\{v_3,v_4,v_5\}_T)-g(\{v_1,v_2,v_3\}_T,v_4,v_5)+g(\{v_1,v_2,v_4\}_T,v_3,v_5)\\
&+g(\{v_1,v_2,v_5\}_T,v_3,v_4)-g(\{v_1,v_2,v_5\}_T,v_4,v_3)+l_T(v_1,v_2)g(v_3,v_4,v_5)\\
&-r_T(v_4,v_5)g(v_1,v_2,v_3)+r_T(v_3,v_5)g(v_1,v_2,v_4)+r_T(v_3,v_4)g(v_1,v_2,v_5)\\
&-r_T(v_4,v_3)g(v_1,v_2,v_5)\\
=&g(v_1,v_2,l(Tv_3,Tv_4)v_5)+g(v_1,v_2,m(Tv_3,Tv_5)v_4)+g(v_1,v_2,r(Tv_4,Tv_5)v_3)\\
&-g(l(Tv_1,Tv_2)v_3, v_4,v_5)-g(m(Tv_1,Tv_3)v_2, v_4,v_5)-g(r(Tv_2,Tv_3)v_1, v_4,v_5)\\
&+g(l(Tv_1,Tv_2)v_4, v_3,v_5)+g(m(Tv_1,Tv_4)v_2, v_3,v_5)+g(r(Tv_2,Tv_4)v_1, v_3,v_5)\\
&+g(l(Tv_1,Tv_2)v_5, v_3,v_4)+g(m(Tv_1,Tv_5)v_2, v_3,v_4)+g(r(Tv_2,Tv_5)v_1, v_3,v_4)\\
&-g(l(Tv_1,Tv_2)v_5, v_4,v_3)-g(m(Tv_1,Tv_5)v_2, v_4,v_3)-g(r(Tv_2,Tv_5)v_1, v_4,v_3)\\
&+\{Tv_1,Tv_2,g(v_3,v_4,v_5)\}-T\Big(r(Tv_2,g(v_3,v_4,v_5))v_1+m(Tv_1,g(v_3,v_4,v_5))v_2\Big)\\
&-\{g(v_1,v_2,v_3),Tv_4,Tv_5\}+T\Big(l(g(v_1,v_2,v_3),Tv_4)v_5+m(g(v_1,v_2,v_3),Tv_5)v_4\Big)\\
&+\{g(v_1,v_2,v_4),Tv_3,Tv_5\}-T\Big(l(g(v_1,v_2,v_4),Tv_3)v_5+m(g(v_1,v_2,v_4),Tv_5)v_3\Big)\\
&+\{g(v_1,v_2,v_5),Tv_3,Tv_4\}-T\Big(l(g(v_1,v_2,v_5),Tv_3)v_4+m(g(v_1,v_2,v_5),Tv_4)v_3\Big)\\
&-\{g(v_1,v_2,v_5),Tv_4,Tv_3\}+T\Big(l(g(v_1,v_2,v_5),Tv_4)v_3+m(g(v_1,v_2,v_5),Tv_3)v_4\Big),
\end{align*}
\begin{align*}
&\partial_2^3g(v_1,v_2,v_3,v_4,v_5)\\
=&g(v_1,\{v_2,v_3,v_4\}_T,v_5)-g(\{v_1,v_2,v_3\}_T,v_4,v_5)+g(\{v_1,v_3,v_2\}_T,v_4,v_5)\\
&+g(\{v_1,v_4,v_2\}_T,v_3,v_5)-g(\{v_1,v_4,v_3\}_T,v_2,v_5)+m_T(v_1,v_5)g(v_2,v_3,v_4)\\
 &-r_T(v_4,v_5)g(v_1,v_2,v_3)+r_T(v_4,v_5)g(v_1,v_3,v_2)+r_T(v_3,v_5)g(v_1,v_4,v_2)\\
  &-r_T(v_2,v_5)g(v_1,v_4,v_3)\\
   =&g(v_1,l(Tv_2,Tv_3)v_4,v_5)+g(v_1,m(Tv_2,Tv_4)v_3,v_5)+g(v_1,r(Tv_3,Tv_4)v_2,v_5)\\
    &-g(l(Tv_1,Tv_2)v_3, v_4,v_5)-g(m(Tv_1,Tv_3)v_2, v_4,v_5)-g(r(Tv_2,Tv_3)v_1, v_4,v_5)\\
     &+g(l(Tv_1,Tv_3)v_2, v_4,v_5)+g(m(Tv_1,Tv_2)v_3, v_4,v_5)+g(r(Tv_3,Tv_2)v_1, v_4,v_5)\\
      &+g(l(Tv_1,Tv_4)v_2, v_3,v_5)+g(m(Tv_1,Tv_2)v_4, v_3,v_5)+g(r(Tv_4,Tv_2)v_1, v_3,v_5)\\
      &-g(l(Tv_1,Tv_4)v_3, v_2,v_5)-g(m(Tv_1,Tv_3)v_4, v_2,v_5)-g(r(Tv_4,Tv_3)v_1, v_2,v_5)\\
 &+\{Tv_1,g(v_2,v_3,v_4),Tv_5\}-T\Big(l(Tv_1,g(v_2,v_3,v_4))v_5+r(g(v_2,v_3,v_4),Tv_5)v_1\Big)\\
 &-\{g(v_1,v_2,v_3),Tv_4,Tv_5\}+T\Big(l(g(v_1,v_2,v_3),Tv_4)v_5+m(g(v_1,v_2,v_3),Tv_5)v_4\Big)\\
  &+\{g(v_1,v_3,v_2),Tv_4,Tv_5\}-T\Big(l(g(v_1,v_3,v_2),Tv_4)v_5+m(g(v_1,v_3,v_2),Tv_5)v_4\Big)\\
  &+\{g(v_1,v_4,v_2),Tv_3,Tv_5\}-T\Big(l(g(v_1,v_4,v_2),Tv_3)v_5+m(g(v_1,v_4,v_2),Tv_5)v_3\Big)\\
  &-\{g(v_1,v_4,v_3),Tv_2,Tv_5\}+T\Big(l(g(v_1,v_4,v_3),Tv_2)v_5+m(g(v_1,v_4,v_3),Tv_5)v_2\Big),
\end{align*}

%\begin{definition}
%Let $V$ be an $\mathfrak{L}$-module and $T$ a relative Rota-Baxter operator on a Leibniz triple system $(\mathfrak{L},\{\cdot,\cdot,\cdot\})$ with respect to a representation $(r,m,l).$ The cochain complex $C^n_T(V,\mathfrak{L})=C^n(V,\mathfrak{L}),$ $n\geq 1.$ The coboundary operator $d^{2n-1}:C^{2n-1}_T(V,\mathfrak{L})\rightarrow C^{2n+1}_T(V,\mathfrak{L})$ defined by $d^1=\partial^1$ and $d^3_i=\partial_i^3(i=1,2).$
%\end{definition}

The set $Z^1_T(V,\mathfrak{L})=\{f\in C^1_T(V,\mathfrak{L})~|~\partial^1f=0\}$ is called a space of $1$-cocycles. We call the set $Z^3_T(V,\mathfrak{L})=\{f\in C^3_T(V,\mathfrak{L})~|~\partial_i^3f=0~(i=1,2)\}$ is a space of $3$-cocycles. And the set $B^3_T(V,\mathfrak{L})=\{\partial^1f~|~f\in C^1_T(V,\mathfrak{L})\}$ is called the $3$-coboundaries.

We define the set
\begin{align*}
H^1_T(V,\mathfrak{L})=&Z^1_T(V,\mathfrak{L}),\\
H^3_T(V,\mathfrak{L})=&Z^3_T(V,\mathfrak{L})/B^3_T(V,\mathfrak{L}),
\end{align*}
is the $1$-th and $3$-th cohomology group for the relative Rota-Baxter operator $T.$

For any $a,b\in \mathfrak{L},$ we define $\delta(a,b): V\rightarrow \mathfrak{L}$ by
\begin{align*}
\delta(a,b)v=Tr(a,b)v-Tr(b,a)v-R_{(a,b)}Tv+R_{(b,a)}Tv,
\end{align*}
for any $v\in V,$ where $R_{(a,b)}v=\{v,a,b\}.$

\begin{remark}\label{rem284.3}
By {\rm \cite[Corollary 13]{Bremner}}, we know that $R_{(b,a)}-R_{(a,b)}$ is a derivation of a Leibniz triple system.
\end{remark}

\begin{proposition}
Let $T$ be a relative Rota-Baxter operator on a Leibniz triple system $(\mathfrak{L},\{\cdot,\cdot,\cdot\})$ with respect to a representation $(r,m,l).$ Then $\delta(a,b)$ is a $1$-cocycle on the Leibniz triple system $(V,\{\cdot,\cdot,\cdot\}_T)$ with coefficient in $(r_T,m_T,l_T).$
\end{proposition}
\begin{proof}
For any $x_1,x_2,x_3\in V,$ we have
\small{\begin{align*}
&\partial^1\delta(a,b)(x_1,x_2,x_3)\\
=&\{Tr(a,b)x_1,Tx_2,Tx_3\}-\{Tr(b,a)x_1,Tx_2,Tx_3\}-\{R_{(a,b)}Tx_1,Tx_2,Tx_3\}+\{R_{(b,a)}Tx_1,Tx_2,Tx_3\}\\
&-Tl(Tr(a,b)x_1,Tx_2)x_3+Tl(Tr(b,a)x_1,Tx_2)x_3+Tl(R_{(a,b)}Tx_1,Tx_2)x_3-Tl(R_{(b,a)}Tx_1,Tx_2)x_3\\
&-Tm(Tr(a,b)x_1,Tx_3)x_2+Tm(Tr(b,a)x_1,Tx_3)x_2+Tm(R_{(a,b)}Tx_1,Tx_3)x_2-Tm(R_{(b,a)}Tx_1,Tx_3)x_2\\
&+\{Tx_1,Tr(a,b)x_2,Tx_3\}-\{Tx_1,Tr(b,a)x_2,Tx_3\}-\{Tx_1,R_{(a,b)}Tx_2,Tx_3\}+\{Tx_1,R_{(b,a)}Tx_2,Tx_3\}\\
&-Tl(Tx_1,Tr(a,b)x_2)x_3+Tl(Tx_1,Tr(b,a)x_2)x_3+Tl(Tx_1,R_{(a,b)}Tx_2)x_3-Tl(Tx_1,R_{(b,a)}Tx_2)x_3\\
&-Tr(Tr(a,b)x_2,Tx_3)x_1+Tr(Tr(b,a)x_2,Tx_3)x_1+Tr(R_{(a,b)}Tx_2,Tx_3)x_1-Tr(R_{(b,a)}Tx_2,Tx_3)x_1\\
&+\{Tx_1,Tx_2,Tr(a,b)x_3\}-\{Tx_1,Tx_2,Tr(b,a)x_3\}-\{Tx_1,Tx_2,R_{(a,b)}Tx_3\}+\{Tx_1,Tx_2,R_{(b,a)}Tx_3\}\\
&-Tm(Tx_1,Tr(a,b)x_3)x_2+Tm(Tx_1,Tr(b,a)x_3)x_2+Tm(Tx_1,R_{(a,b)}Tx_3)x_2-Tm(Tx_1,R_{(b,a)}Tx_3)x_2\\
&-Tr(Tx_2,Tr(a,b)x_3)x_1+Tr(Tx_2,Tr(b,a)x_3)x_1+Tr(Tx_2,R_{(a,b)}Tx_3)x_1-Tr(Tx_2,R_{(b,a)}Tx_3)x_1\\
&-Tr(a,b)l(Tx_1,Tx_2)x_3+Tr(b,a)l(Tx_1,Tx_2)x_3+R_{(a,b)}Tl(Tx_1,Tx_2)x_3-R_{(b,a)}Tl(Tx_1,Tx_2)x_3\\
&-Tr(a,b)m(Tx_1,Tx_3)x_2+Tr(b,a)m(Tx_1,Tx_3)x_2+R_{(a,b)}Tm(Tx_1,Tx_3)x_2-R_{(b,a)}Tm(Tx_1,Tx_3)x_2\\
&-Tr(a,b)r(Tx_2,Tx_3)x_1+Tr(b,a)r(Tx_2,Tx_3)x_1+R_{(a,b)}Tr(Tx_2,Tx_3)x_1-R_{(b,a)}Tr(Tx_2,Tx_3)x_1\\
\overset{(\ref{282.3})}{=}
&Tr(Tx_2,Tx_3)r(a,b)x_1-Tr(Tx_2,Tx_3)r(b,a)x_1+Tm(Tx_1,Tx_3)r(a,b)x_2-Tm(Tx_1,Tx_3)r(b,a)x_2\\
&+Tl(Tx_1,Tx_2)r(a,b)x_3-Tl(Tx_1,Tx_2)r(b,a)x_3-\{R_{(a,b)}Tx_1,Tx_2,Tx_3\}+\{R_{(b,a)}Tx_1,Tx_2,Tx_3\}\\
&+Tm(R_{(a,b)}Tx_1,Tx_3)x_2-Tm(R_{(b,a)}Tx_1,Tx_3)x_2-\{Tx_1,R_{(a,b)}Tx_2,Tx_3\}+\{Tx_1,R_{(b,a)}Tx_2,Tx_3\}\\
&+Tl(Tx_1,R_{(a,b)}Tx_2)x_3-Tl(Tx_1,R_{(b,a)}Tx_2)x_3-Tr(Tr(a,b)x_2,Tx_3)x_1+Tr(Tr(b,a)x_2,Tx_3)x_1\\
&+Tr(R_{(a,b)}Tx_2,Tx_3)x_1-Tr(R_{(b,a)}Tx_2,Tx_3)x_1-\{Tx_1,Tx_2,R_{(a,b)}Tx_3\}+\{Tx_1,Tx_2,R_{(b,a)}Tx_3\}\\
&+Tm(Tx_1,R_{(a,b)}Tx_3)x_2-Tm(Tx_1,R_{(b,a)}Tx_3)x_2+Tr(Tx_2,R_{(a,b)}Tx_3)x_1-Tr(Tx_2,R_{(b,a)}Tx_3)x_1\\
&-Tr(a,b)l(Tx_1,Tx_2)x_3+Tr(b,a)l(Tx_1,Tx_2)x_3+R_{(a,b)}Tl(Tx_1,Tx_2)x_3-R_{(b,a)}Tl(Tx_1,Tx_2)x_3\\
&-Tr(a,b)m(Tx_1,Tx_3)x_2+Tr(b,a)m(Tx_1,Tx_3)x_2+R_{(a,b)}Tm(Tx_1,Tx_3)x_2-R_{(b,a)}Tm(Tx_1,Tx_3)x_2\\
&-Tr(a,b)r(Tx_2,Tx_3)x_1+Tr(b,a)r(Tx_2,Tx_3)x_1+R_{(a,b)}Tr(Tx_2,Tx_3)x_1-R_{(b,a)}Tr(Tx_2,Tx_3)x_1\\
\overset{Rem.\ref{rem284.3}}{=}
&Tr(Tx_2,Tx_3)r(a,b)x_1-Tr(Tx_2,Tx_3)r(b,a)x_1+Tm(Tx_1,Tx_3)r(a,b)x_2-Tm(Tx_1,Tx_3)r(b,a)x_2\\
&+Tl(Tx_1,Tx_2)r(a,b)x_3-Tl(Tx_1,Tx_2)r(b,a)x_3+Tm(R_{(a,b)}Tx_1,Tx_3)x_2-Tm(R_{(b,a)}Tx_1,Tx_3)x_2\\
&+Tl(Tx_1,R_{(a,b)}Tx_2)x_3-Tl(Tx_1,R_{(b,a)}Tx_2)x_3-Tr(Tr(a,b)x_2,Tx_3)x_1+Tr(Tr(b,a)x_2,Tx_3)x_1\\
&+Tr(R_{(a,b)}Tx_2,Tx_3)x_1-Tr(R_{(b,a)}Tx_2,Tx_3)x_1+Tm(Tx_1,R_{(a,b)}Tx_3)x_2-Tm(Tx_1,R_{(b,a)}Tx_3)x_2\\
&+Tr(Tx_2,R_{(a,b)}Tx_3)x_1-Tr(Tx_2,R_{(b,a)}Tx_3)x_1-Tr(a,b)l(Tx_1,Tx_2)x_3+Tr(b,a)l(Tx_1,Tx_2)x_3\\
&-Tr(a,b)m(Tx_1,Tx_3)x_2+Tr(b,a)m(Tx_1,Tx_3)x_2-Tr(a,b)r(Tx_2,Tx_3)x_1+Tr(b,a)r(Tx_2,Tx_3)x_1\\
\overset{Def.\ref{def282,3}}{=}&0.
\end{align*}}
Hence, we have $\partial^1\delta(a,b)(x_1,x_2,x_3)=0,$ which implies that $\delta(a,b)$ is a $1$-cocycle.
\end{proof}

We can use these cohomology groups to characterize linear and formal deformation of relative Rota-Baxter operators in the following section.
\section{Deformations of relative Rota-Baxter operators on Leibniz triple systems}

In this section, we will deal with the linear deformations of relative Rota-Baxter operators on Leibniz triple systems and we show that the infinitesimals of two equivalent linear deformations of a relative Rota-Baxter operators on a Leibniz triple system are in the same cohomology classes of the first cohomology group.

\begin{definition}
Let $T$ be a relative Rota-Baxter operator on a Leibniz triple system $(\mathfrak{L},\{\cdot,\cdot,\cdot\})$ with respect to a representation $(r,m,l)$ and $\mathfrak{T}:V\rightarrow \mathfrak{L}$ a linear map. If $T_t=T+t\mathfrak{T}$ are still  relative Rota-Baxter operators on $(\mathfrak{L},\{\cdot,\cdot,\cdot\})$ with respect to a representation $(r,m,l)$ for all $t$, we say that $\mathfrak{T}$ generates a linear deformation of the relative Rota-Baxter operator.
\end{definition}

Suppose $\mathfrak{T}$ generates a linear deformation of the relative Rota-Baxter operator $T$. Then we have
\begin{align*}
\{T_tu,T_tv,T_tw\}=T_t\Big(l(T_tu,T_tv)w+m(T_tu,T_tw)v+r(T_tv,T_tw)u\Big),~~\forall~~u,v,w\in V.
\end{align*}
This is equivalent to the following equations

\begin{gather}
\begin{aligned}\label{284.11}
&\{Tu,Tv,\mathfrak{T}w\}+\{Tu,\mathfrak{T}v,Tw\}+\{\mathfrak{T}u,Tv,Tw\}\\
=&\mathfrak{T}\Big(l(Tu,Tv)w+m(Tu,Tw)v+r(Tv,Tw)u\Big)\\
&+T\Big(l(Tu,\mathfrak{T}v)w+m(\mathfrak{T}u,Tw)v+r(Tv,\mathfrak{T}w)u\\
&+l(\mathfrak{T}u,Tv)w+m(Tu,\mathfrak{T}w)v+r(\mathfrak{T}v,Tw)u\Big),
\end{aligned}
\end{gather}
and
\begin{align*}
&\{\mathfrak{T}u,\mathfrak{T}v,Tw\}+\{Tu,\mathfrak{T}v,\mathfrak{T}w\}+\{\mathfrak{T}u,Tv,\mathfrak{T}w\}\\
=&T\Big(l(\mathfrak{T}u,\mathfrak{T}v)w+m(\mathfrak{T}u,\mathfrak{T}w)v+r(\mathfrak{T}v,\mathfrak{T}w)u\Big)
+\mathfrak{T}\Big(l(Tu,\mathfrak{T}v)w+m(\mathfrak{T}u,Tw)v+r(Tv,\mathfrak{T}w)u\\
&+l(\mathfrak{T}u,Tv)w+m(Tu,\mathfrak{T}w)v+r(\mathfrak{T}v,Tw)u\Big).
\end{align*}
Note that Eq. (\ref{284.11}) means that $\mathfrak{T}\in C^1_T(V,\mathfrak{L})$ is a 1-cocycle.

We also have
\begin{align*}
\{\mathfrak{T}u,\mathfrak{T}v,\mathfrak{T}w\}
=\mathfrak{T}\Big(l(\mathfrak{T}u,\mathfrak{T}v)w+m(\mathfrak{T}u,\mathfrak{T}w)v+r(\mathfrak{T}v,\mathfrak{T}w)u\Big),
\end{align*}
which means that $\mathfrak{T}$ is a relative Rota-Baxter operator on the Leibniz triple system $(\mathfrak{L},\{\cdot,\cdot,\cdot\})$ with respect to a representation $(r,m,l)$.
\begin{definition}
Let $T$ and $T'$ be two relative Rota-Baxter operators on a Leibniz triple system $(\mathfrak{L},\{\cdot,\cdot,\cdot\})$ with respect to a representation $(r,m,l)$. A homomorphism from $T'$ to $T$ consists of a homomorphism $\phi_{\mathfrak{L}}:\mathfrak{L}\rightarrow \mathfrak{L}$ and a linear map $\phi_{V}:V\rightarrow V$, such that
\begin{align*}
T\circ \phi_{V}=& \phi_{\mathfrak{L}}\circ T',\\
\phi_{V}l(x,y)(v)=&l(\phi_{\mathfrak{L}}(x),\phi_{\mathfrak{L}}(y))\phi_{V}(v),\\
\phi_{V}m(x,y)(v)=&m(\phi_{\mathfrak{L}}(x),\phi_{\mathfrak{L}}(y))\phi_{V}(v),\\
\phi_{V}r(x,y)(v)=&r(\phi_{\mathfrak{L}}(x),\phi_{\mathfrak{L}}(y))\phi_{V}(v),
\end{align*}
for all $x,y\in \mathfrak{L}$ and $v\in V$. In particular, if $\phi_{\mathfrak{L}}$ and $\phi_{V}$ are invertible, then $(\phi_{\mathfrak{L}},\phi_{V})$ is called an isomorphism from $T'$ to $T$.
\end{definition}

\begin{definition}
Let $T:V\rightarrow \mathfrak{L}$ be a relative Rota-Baxter operator on a Leibniz triple system $(\mathfrak{L},\{\cdot,\cdot,\cdot\})$ with respect to a representation $(r,m,l)$.

{\rm (i)} Two linear deformations $T_t^1=T+t\mathfrak{T}_1$ and $T_t^2=T+t\mathfrak{T}_2$ are said to be equivalent if there exists $x,y\in \mathfrak{L}$ such that $\Big({\rm id}_\mathfrak{L}+t(R_{(x,y)}-R_{(y,x)}),
{\rm id}_V+t(r(x,y)-r(y,x))\Big)$ is a homomorphism from $T_t^2$ to $T_t^1$.

{\rm(ii)} A linear deformation $T_t=T+t\mathfrak{T}$ of a relative Rota-Baxter operator $T$ is said to be trivial if there exists $x,y\in \mathfrak{L}$ such that $\Big({\rm id}_\mathfrak{L}+t(R_{(x,y)}-R_{(y,x)}),
{\rm id}_V+t(r(x,y)-r(y,x))\Big)$ is a homomorphism from $T_t$ to $T$.
\end{definition}

Let $\Big({\rm id}_\mathfrak{L}+t(R_{(x,y)}-R_{(y,x)}),
{\rm id}_V+t(r(x,y)-r(y,x))\Big)$ be a homomorphism from $T_t^2$ to $T_t^1$. Then ${\rm id}_\mathfrak{L}+t(R_{(x,y)}-R_{(y,x)})$ is a homomorphism of $\mathfrak{L}$, we have
\begin{align*}
&\{a,(R_{(x,y)}-R_{(y,x)})b,(R_{(x,y)}-R_{(y,x)})c\}+\{(R_{(x,y)}-R_{(y,x)})a,b,(R_{(x,y)}-R_{(y,x)})c\}\notag\\
&+\{(R_{(x,y)}-R_{(y,x)})a,(R_{(x,y)}-R_{(y,x)})b,c\}=0,
\end{align*}
and
\begin{align*}
\{(R_{(x,y)}-R_{(y,x)})a,(R_{(x,y)}-R_{(y,x)})b,(R_{(x,y)}-R_{(y,x)})c\}=0.
\end{align*}

By $T_t^1\Big({\rm id}_V+t(r(x,y)-r(y,x))\Big)(v)=\Big({\rm id}_\mathfrak{L}+t(R_{(x,y)}-R_{(y,x)})\Big)T_t^2(v)$, we have
\begin{align}
(\mathfrak{T}_2-\mathfrak{T}_1)(v)=T(r(x,y)-r(y,x))(v)-(R_{(x,y)}-R_{(y,x)})T(v),\label{284.6}
\end{align}
\begin{align*}
\mathfrak{T}_1(r(x,y)-r(y,x))(v)=(R_{(x,y)}-R_{(y,x)})\mathfrak{T}_2(v).
\end{align*}

By
\begin{align*}
&\Big({\rm id}_V+t(r(x,y)-r(y,x))\Big)r(a,b)v\\
=&r\Big(\big({\rm id}_\mathfrak{L}+t(R_{(x,y)}-R_{(y,x)})\big)a,\big({\rm id}_\mathfrak{L}+t(R_{(x,y)}-R_{(y,x)})\big)b\Big)\Big({\rm id}_V+t(r(x,y)-r(y,x))\Big)v,\\
&\Big({\rm id}_V+t(r(x,y)-r(y,x))\Big)m(a,b)v\\
=&m\Big(\big({\rm id}_\mathfrak{L}+t(R_{(x,y)}-R_{(y,x)})\big)a,\big({\rm id}_\mathfrak{L}+t(R_{(x,y)}-R_{(y,x)})\big)b\Big)\Big({\rm id}_V+t(r(x,y)-r(y,x))\Big)v,\\
&\Big({\rm id}_V+t(r(x,y)-r(y,x))\Big)l(a,b)v\\
=&l\Big(\big({\rm id}_\mathfrak{L}+t(R_{(x,y)}-R_{(y,x)})\big)a,\big({\rm id}_\mathfrak{L}+t(R_{(x,y)}-R_{(y,x)})\big)b\Big)\Big({\rm id}_V+t(r(x,y)-r(y,x))\Big)v,
\end{align*}
we have,
\begin{align*}
&r\big(a,(R_{(x,y)}-R_{(y,x)})b\big)\big((r(x,y)-r(y,x))v\big)
+r\big((R_{(x,y)}-R_{(y,x)})a,b\big)\big((r(x,y)-r(y,x))v\big)\\
&+r\big((R_{(x,y)}-R_{(y,x)})a,(R_{(x,y)}-R_{(y,x)})b\big)v=0,\\
&m\big(a,(R_{(x,y)}-R_{(y,x)})b\big)\big((r(x,y)-r(y,x))v\big)
+m\big((R_{(x,y)}-R_{(y,x)})a,b\big)\big((r(x,y)-r(y,x))v\big)\\
&+m\big((R_{(x,y)}-R_{(y,x)})a,(R_{(x,y)}-R_{(y,x)})b\big)v=0,\\
&l\big(a,(R_{(x,y)}-R_{(y,x)})b\big)\big((r(x,y)-r(y,x))v\big)
+l\big((R_{(x,y)}-R_{(y,x)})a,b\big)\big((r(x,y)-r(y,x))v\big)\\
&+l\big((R_{(x,y)}-R_{(y,x)})a,(R_{(x,y)}-R_{(y,x)})b\big)v=0.
\end{align*}

Note that Eq. (\ref{284.6}) means that there exist $x,y\in \mathfrak{L}$, such that $\mathfrak{T}_2-\mathfrak{T}_1=\delta(x,y)$. Thus we have the following result.
\begin{theorem}
Let $T:V\rightarrow \mathfrak{L}$ be a relative Rota-Baxter operator on a Leibniz triple system $(\mathfrak{L},\{\cdot,\cdot,\cdot\})$ with respect to a representation $(r,m,l)$. If two linear deformations $T_t^1=T+t\mathfrak{T}_1$ and $T_t^2=T+t\mathfrak{T}_2$ of $T$ are equivalent, then $\mathfrak{T}_1$ and $\mathfrak{T}_2$ are in the same class of the cohomology group $H_T^1(V,\mathfrak{L})$.
\end{theorem}

\subsection{Formal deformations of a relative Rota-Baxter operator on Leibniz triple systems}

 Let $\mathbb{F}[[t]]$ be the ring of power series in one variable $t.$ For any $\mathbb{F}$-linear space $V,$ we let $V[[t]]$
denote the vector space of formal power series in $t$ with coefficient in $V.$ If in addition, $(\mathfrak{L},\{\cdot,\cdot,\cdot\})$ is a Leibniz triple system over $\mathbb{F},$ then there is a Leibniz triple system structure over the ring $\mathbb{F}[[t]]$ on $\mathfrak{L}[[t]]$ given by
\begin{align}
\Big\{\sum\limits_{i=0}^{+\infty}x_it^i,\sum\limits_{j=0}^{+\infty}y_jt^j,\sum\limits_{k=0}^{+\infty}z_kt^k\Big\}
=\sum\limits_{s=0}^{+\infty}\sum\limits_{i+j+k=s}^{}\{x_i,y_j,z_k\}t^s,\label{284.1}
\end{align}
for all $x_i,y_j,z_k\in \mathfrak{L}.$

For any representation $(r,m,l)$ of $(\mathfrak{L},\{\cdot,\cdot,\cdot\})$ on $V,$ there is a natural representation of the Leibniz triple system  $\mathfrak{L}[[t]]$ on the $\mathbb{F}[[t]]$-module $V[[t]],$ which is given by
\begin{align}\label{284.2}
l(\sum\limits_{i=0}^{+\infty}x_it^i,\sum\limits_{j=0}^{+\infty}y_jt^j)(\sum\limits_{k=0}^{+\infty}v_kt^k)
=&\sum\limits_{s=0}^{+\infty}\sum\limits_{i+j+k=s}^{}l(x_i,y_j)v_kt^s,\notag\\
m(\sum\limits_{i=0}^{+\infty}x_it^i,\sum\limits_{j=0}^{+\infty}y_jt^j)(\sum\limits_{k=0}^{+\infty}v_kt^k)
=&\sum\limits_{s=0}^{+\infty}\sum\limits_{i+j+k=s}^{}m(x_i,y_j)v_kt^s,\\
r(\sum\limits_{i=0}^{+\infty}x_it^i,\sum\limits_{j=0}^{+\infty}y_jt^j)(\sum\limits_{k=0}^{+\infty}v_kt^k)
=&\sum\limits_{s=0}^{+\infty}\sum\limits_{i+j+k=s}^{}r(x_i,y_j)v_kt^s,\notag
\end{align}
for all $x_i,y_j\in \mathfrak{L}$ and $v_k \in V.$

Let $T$ be a relative Rota-Baxter operator on a Leibniz triple system $(\mathfrak{L},\{\cdot,\cdot,\cdot\})$ with respect to a representation  $(r,m,l).$ Consider a power series
\begin{align}
T_t=\sum\limits_{i=0}^{+\infty}\mathfrak{T}_it^i,~~\mathfrak{T}_i\in {\rm Hom}_{\mathbb{F}}(V,\mathfrak{L}),\label{284.3}
\end{align}
that is, $T_t\in {\rm Hom}_{\mathbb{F}}(V,\mathfrak{L})[[t]]={\rm Hom}_{\mathbb{F}}(V,\mathfrak{L}[[t]]).$ Extend it to be a $\mathbb{F}[[t]]$-module map from $V[[t]]$ to $\mathfrak{L}[[t]],$ which is still denoted by $T_t.$

\begin{definition}
If $T_t=\sum\limits_{i=0}^{+\infty}\mathfrak{T}_it^i$ with $\mathfrak{T}_0=T$ satisfies
\begin{align}
\{T_t(u),T_t(v),T_t(w)\}=T_t(l(T_t(u),T_t(v))w+m(T_t(u),T_t(w))v+r(T_t(v),T_t(w))u),\label{284.4}
\end{align}
we say that $T_t$ is a formal deformation of the relative Rota-Baxter operator $T.$
\end{definition}

Recall that a formal deformation of a Leibniz triple system $(\mathfrak{L},\{\cdot,\cdot,\cdot\})$ is a power series $\mu_t=\sum\limits_{k=0}^{+\infty}\mu_kt^k,$ where $\mu_k\in {\rm Hom}(\otimes^3\mathfrak{L},\mathfrak{L}),$ such that, $\mu_0(x,y,z)=\{x,y,z\},$ for any $x,y,z\in \mathfrak{L}$ and $\mu_t$ defines a Leibniz triple system structure over the ring $\mathbb{F}[[t]]$ on $\mathfrak{L}[[t]].$

Based on the relationship between relative Rota-Baxter operators and Leibniz triple systems, we have

\begin{proposition}
If $T_t=\sum\limits_{i=0}^{+\infty}\mathfrak{T}_it^i$ is a formal deformation of a relative Rota-Baxter operator $T$ on a Leibniz triple system $(\mathfrak{L},\{\cdot,\cdot,\cdot\})$ with respect to a representation $(r,m,l),$ and $\{\cdot,\cdot,\cdot\}_{T_t}$ defined by
\begin{align*}
\{u,v,w\}_{T_t}=\sum\limits_{k=0}^{+\infty}\sum\limits_{i+j=k}^{}
\Big(l(\mathfrak{T}_iu,\mathfrak{T}_jv)w+m(\mathfrak{T}_ju,\mathfrak{T}_iw)v+r(\mathfrak{T}_iv,\mathfrak{T}_jw)u\Big)t^k,
~~\forall u,v,w\in V,
\end{align*}
is a formal deformation of the Leibniz triple system $(V,\{\cdot,\cdot,\cdot\}_T)$ given in Lemma {\rm \ref{lem282.7}}.
\end{proposition}

Applying Eqs. (\ref{284.1})-(\ref{284.3}) to expand Eq. (\ref{284.4}) and collecting coefficients of $t^s,$ we see that Eq. (\ref{284.4}) is equivalent to the system of equations
\begin{align}\label{284.5}
\sum\limits_{i+j+k=s\atop i,j,k\geq0}^{}\{\mathfrak{T}_iu,\mathfrak{T}_jv,\mathfrak{T}_kw\}
=\sum\limits_{i+j+k=s\atop i,j,k\geq0}^{}
\mathfrak{T}_i\Big(l(\mathfrak{T}_ju,\mathfrak{T}_kv)w+m(\mathfrak{T}_ku,\mathfrak{T}_jw)v+r(\mathfrak{T}_jv,\mathfrak{T}_kw)u\Big),
\end{align}
for all $s\geq 0,$ $u,v,w\in V.$

\begin{proposition}\label{prop284.3}
Let $T_t=\sum\limits_{i=0}^{+\infty}\mathfrak{T}_it^i$ be a formal deformation of a relative Rota-Baxter operator $T$ on a Leibniz triple system $(\mathfrak{L},\{\cdot,\cdot,\cdot\})$ with respect to a representation $(r,m,l).$ Then $\mathfrak{T}_1$ is a $1$-cocycle for the relative Rota-Baxter operator $T,$ that is $\partial^1\mathfrak{T}_1=0.$
\end{proposition}
\begin{proof}
For $s=1,$ Eq. (\ref{284.5}) is equivalent to
\begin{align*}
&\{\mathfrak{T}_1u,Tv,Tw\}+\{Tu,\mathfrak{T}_1v,Tw\}+\{Tu,Tv,\mathfrak{T}_1w\}\\
=~&\mathfrak{T}_1\Big(l(Tu,Tv)w+m(Tu,Tw)v+r(Tv,Tw)u\Big)\\
&+T\Big(l(\mathfrak{T}_1u,Tv)w+m(Tu,\mathfrak{T}_1w)v+r(\mathfrak{T}_1v,Tw)u\Big)\\
&+T\Big(l(Tu,\mathfrak{T}_1v)w+m(\mathfrak{T}_1u,Tw)v+r(Tv,\mathfrak{T}_1w)u\Big),
\end{align*}
for any $u,v,w\in V,$ which implies that $\mathfrak{T}_1$ is a $1$-cocycle.
\end{proof}

\begin{definition}
Let $T$ be a relative Rota-Baxter operator on a Leibniz triple system $(\mathfrak{L},\{\cdot,\cdot,\cdot\})$ with respect to a representation $(r,m,l).$ The $1$-cocycle $\mathfrak{T}_1$  is called the infinitesimal of the formal deformation $T_t=\sum\limits_{i=0}^{+\infty}\mathfrak{T}_it^i$ of $T.$
\end{definition}

In the sequel, we discusses equivalent formal deformations.

\begin{definition}\label{def284.5}
Two formal deformations $T_t^\prime=\sum\limits_{i=0}^{+\infty}\mathfrak{T}_i^\prime t^i$ and $T_t=\sum\limits_{i=0}^{+\infty}\mathfrak{T}_it^i$ of a relative Rota-Baxter operator $T=\mathfrak{T}_0^\prime=\mathfrak{T}_0$ on a Leibniz triple system $(\mathfrak{L},\{\cdot,\cdot,\cdot\})$ with respect to a representation $(r,m,l)$ are said to be equivalent if there exist $x,y\in \mathfrak{L},$ $\phi_i\in {\rm End}(\mathfrak{L})$ and $\psi_i\in {\rm End}(V),$ $i\geq2$ such that for
\begin{align*}
\phi_t=~&{\rm id}_\mathfrak{L}+t(R_{(x,y)}-R_{(y,x)})+\sum\limits_{i=2}^{+\infty}\phi_it^i,\\
\psi_t=~&{\rm id}_V+t(r(x,y)-r(y,x))+\sum\limits_{i=2}^{+\infty}\psi_it^i,
\end{align*}
the following conditions hold:
\begin{align*}
{\rm (1)}&~\{\phi_t(x),\phi_t(y),\phi_t(z)\}=\phi_t\{x,y,z\},{\rm~for~all}~x,y,z\in \mathfrak{L};\\
{\rm (2)}&~\psi_tl(x,y)u=l(\phi_t(x),\phi_t(y))\psi_t(u),\\
         &~\psi_tm(x,y)u=m(\phi_t(x),\phi_t(y))\psi_t(u),\\
         &~\psi_tr(x,y)u=r(\phi_t(x),\phi_t(y))\psi_t(u),{\rm~for~all}~x,y\in \mathfrak{L},~u\in V;\\
{\rm (3)}&~T_t\circ \psi_t= \phi_t \circ T_t^\prime ~{\rm as~} \mathbb{F}[[t]]{\rm -module~maps}.
\end{align*}
\end{definition}

\begin{theorem}
If two formal deformations of a relative Rota-Baxter operator $T$ on a Leibniz triple system $(\mathfrak{L},\{\cdot,\cdot,\cdot\})$ with respect to a representation $(r,m,l)$ are equivalent, then their infinitesimals are in the same cohomology class.
\end{theorem}
\begin{proof}
Let $(\phi_t,\psi_t)$ be two maps defined by Definition \ref{def284.5}, which gives an equivalence between two deformations $T^\prime_t=\sum\limits_{i=0}^{+\infty}\mathfrak{T}_it^i$ and $T_t=\sum\limits_{i=0}^{+\infty}\mathfrak{T}_it^i$ of a relative Rota-Baxter operator $T.$ By $\phi_t\circ T^\prime_t=T_t\circ \psi_t,$ we have
\begin{align*}
\mathfrak{T}^\prime_1 v- \mathfrak{T}_1v=T(r(a,b)-r(b,a))v-(R_{(a,b)}-R_{(b,a)})Tv=\delta(a,b)v,~\forall v\in V,
\end{align*}
which implies that $\mathfrak{T}^\prime_1$ and $\mathfrak{T}_1$ are in the same cohomology class.
\end{proof}

\subsection{Order $n$ deformations of a relative Rota-Baxter operator on  Leibniz triple systems.}

We introduce a cohomology class associated to any order $n$ deformation of a relative Rota-Baxter operator, and show that an order $n$ deformation of a relative Rota-Baxter operator is extensible if and only if ${\rm Ob}_T=-\partial^1\mathfrak{T}_{n+1}.$ Thus we call this cohomology class the obstruction of an order $n$ deformation being extendable.
\begin{definition}
Let $T$ be a relative Rota-Baxter operator on a Leibniz triple system $(\mathfrak{L},\{\cdot,\cdot,\cdot\})$ with respect to a representation $(r,m,l).$ If $T_t=\sum\limits_{i=0}^{n}\mathfrak{T}_it^i$ with $\mathfrak{T}_0=T,$ $\mathfrak{T}_i\in {\rm Hom}_\mathbb{F}(V,\mathfrak{L})$, $i=2,3,..., n,$ defines on $\mathbb{F}[[t]]/(t^{n+1})$-module map from $V[[t]]/(t^{n+1})$ to the Leibniz triple system $\mathfrak{L}[[t]]/(t^{n+1}),$ for all $u,v,w\in V,$ satisfying
\begin{align}\label{284.8}
\{T_t(u),T_t(v),T_t(w)\}=T_t\Big(l(T_t(u),T_t(v))w+m(T_t(u),T_t(w))v+r(T_t(v),T_t(w))u\Big),
\end{align}
we say that $T_t$ is an order $n$ deformation of the relative Rota-Baxter operator $T.$
\end{definition}

\begin{remark}
Obviously, the left hand side of Eq. (\ref{284.8}) holds in the Leibniz triple system $\mathfrak{L}[[t]]/(t^{n+1})$ and the right hand side makes sense since $T_t$ is an $\mathbb{F}[[t]]/(t^{n+1})$-module map.
\end{remark}

\begin{definition}
Let $T_t=\sum\limits_{i=0}^{n}\mathfrak{T}_it^i$ be an order $n$ deformation of a relative Rota-Baxter operator $T$ on a Leibniz triple system $(\mathfrak{L},\{\cdot,\cdot,\cdot\})$ with respect to a representation $(r,m,l).$ If there exists a $1$-cochain $\mathfrak{T}_{n+1}\in {\rm Hom}_\mathbb{F}(V,\mathfrak{L})$ such that $\widetilde{T}_t=T_t+\mathfrak{T}_{n+1}t^{n+1}$ is an order $n+1$ deformation of the relative Rota-Baxter operator $T,$ then we say that $T_t$ is extendable.
\end{definition}

\begin{theorem}\label{thm284.15}
Let $T:V\rightarrow \mathfrak{L}$ be a relative Rota-Baxter operator on a Leibniz triple system $(\mathfrak{L},\{\cdot,\cdot,\cdot\})$ with respect to a representation $(r,m,l)$,
and $T_t=\sum\limits_{i=0}^{n}\mathfrak{T}_it^i$ be an order $n$ deformation of  $T$. Then $T_t$ is extendable if and only if the cohomology class $[{\rm Ob}_T]\in H^3_T(V,\mathfrak{L})$ is trivial, where ${\rm Ob}_T\in C^3_T(V,\mathfrak{L})$ is defined by
\begin{align*}
{\rm Ob}_T(u,v,w)=\sum\limits_{i+j+k=n+1\atop 0\leq i,j,k\leq n}^{}\Bigg(\{\mathfrak{T}_iu,\mathfrak{T}_jv,\mathfrak{T}_kw\}-\mathfrak{T}_i\Big(l(\mathfrak{T}_ju,\mathfrak{T}_kv)w
+m(\mathfrak{T}_ku,\mathfrak{T}_jw)v+r(\mathfrak{T}_jv,\mathfrak{T}_kw)u\Big)\Bigg),
\end{align*}
for all $u,v,w\in V$.
\end{theorem}
\begin{proof}
Let $\widetilde{T_t}=\sum\limits_{i=0}^{n+1}\mathfrak{T}_it^i$ be the extension of $T_t$, then for all $u,v,w\in V$,
\begin{align}\label{284.17}
\{\widetilde{T_t}(u),\widetilde{T_t}(v),\widetilde{T_t}(w)\}=\widetilde{T_t}\Big(l(\widetilde{T_t}(u),\widetilde{T_t}(v))w
+m(\widetilde{T_t}(u),\widetilde{T_t}(w))v+r(\widetilde{T_t}(v),\widetilde{T_t}(w))u\Big),
\end{align}
Expanding the Eq. (\ref{284.17}) and comparing the coefficients of $t^n$ yields that
\begin{align*}
\sum\limits_{i+j+k=n+1\atop  i,j,k\geq 0}^{}\Bigg(\{\mathfrak{T}_iu,\mathfrak{T}_jv,\mathfrak{T}_kw\}-\mathfrak{T}_i\Big(l(\mathfrak{T}_ju,\mathfrak{T}_kv)w
+m(\mathfrak{T}_ku,\mathfrak{T}_jw)v+r(\mathfrak{T}_jv,\mathfrak{T}_kw)u\Big)\Bigg)=0,
\end{align*}
which is equivalent to
\begin{align*}
&\sum\limits_{i+j+k=n+1\atop  i,j,k\geq 1}^{}\Bigg(\{\mathfrak{T}_iu,\mathfrak{T}_jv,\mathfrak{T}_kw\}-\mathfrak{T}_i\Big(l(\mathfrak{T}_ju,\mathfrak{T}_kv)w
+m(\mathfrak{T}_ku,\mathfrak{T}_jw)v+r(\mathfrak{T}_jv,\mathfrak{T}_kw)u\Big)\Bigg)\\
&+\{\mathfrak{T}_{n+1}u,Tv,Tw\}+\{Tu,\mathfrak{T}_{n+1}v,Tw\}+\{Tu,Tv,\mathfrak{T}_{n+1}w\}\\
&-\mathfrak{T}_{n+1}\Big(l(Tu,Tv)w+m(Tu,Tw)v+r(Tv,Tw)u\Big)\\
&-T\Big(l(\mathfrak{T}_{n+1}u,Tv)w+m(\mathfrak{T}_{n+1}u,Tw)v+r(\mathfrak{T}_{n+1}v,Tw)u\Big)\\
&-T\Big(l(Tu,\mathfrak{T}_{n+1}v)w+m(Tu,\mathfrak{T}_{n+1}w)v+r(Tv,\mathfrak{T}_{n+1}w)u\Big)=0,
\end{align*}
which is also equivalent to
\begin{align}\label{284.18}
{\rm Ob}_T+\partial^1(\mathfrak{T}_{n+1})=0.
\end{align}
From Eq. (\ref{284.18}), we get ${\rm Ob}_T=-\partial^1(\mathfrak{T}_{n+1})$. Thus, the cohomology class $[{\rm Ob}_T]$ is trivial.

Conversely, suppose that the cohomology class $[{\rm Ob}_T]$ is trivial, then there exists $\mathfrak{T}_{n+1}\in C^1_T(V,\mathfrak{L})$ such that ${\rm Ob}_T=-\partial^1(\mathfrak{T}_{n+1})$. Set $\widetilde{T_t}=T_t+\mathfrak{T}_{n+1}t^{n+1}$. Then for all $0\leq s\leq n+1$, $\widetilde{T_t}$ satisfies
\begin{align*}
\sum\limits_{i+j+k=s}^{}\Bigg(\{\mathfrak{T}_iu,\mathfrak{T}_jv,\mathfrak{T}_kw\}-\mathfrak{T}_i\Big(l(\mathfrak{T}_ju,\mathfrak{T}_kv)w
+m(\mathfrak{T}_ku,\mathfrak{T}_jw)v+r(\mathfrak{T}_jv,\mathfrak{T}_kw)u\Big)\Bigg)=0,
\end{align*}
which implies that $\widetilde{T_t}$ is an order $n+1$ deformation of $T$. Hence it is an extension of $T_t$.
\end{proof}

\begin{definition}
Let $T:V\rightarrow \mathfrak{L}$ be a relative Rota-Baxter operator on a Leibniz triple system $(\mathfrak{L},\{\cdot,\cdot,\cdot\})$ with respect to a representation $(r,m,l)$ and
$T_t=\sum\limits_{i=0}^{n}\mathfrak{T}_it^i$ be an order $n$ deformation of $T$. Then the cohomology class $[{\rm Ob}_T]\in H^3_T(V,\mathfrak{L})$ defined in Theorem \ref{thm284.15} is called the obstruction of $T_t$ being extensible.
\end{definition}

\begin{corollary}
Let $T:V\rightarrow \mathfrak{L}$ be a relative Rota-Baxter operator on a Leibniz triple system $(\mathfrak{L},\{\cdot,\cdot,\cdot\})$ with respect to a representation $(r,m,l)$. If $H^3_T(V,\mathfrak{L})=0,$ then every $1$-cocycle in $Z^1_T(V,\mathfrak{L})$ is the infinitesimal of some formal deformation of the relative Rota-Baxter operator $T$.
\end{corollary}

\section{From cohomology groups of relative Rota-Baxter operators on Leibniz algebras to those on Leibniz triple systems}

Motivated by the construction of Leibniz triple systems from Leibniz algebras. We give some connection between relative Rota-Baxter operators on Leibniz algebras and Leibniz triple systems.
\begin{definition}\cite{LP}
A right Leibniz algebra is a vector space $L$ together with a bilinear operation $[\cdot,\cdot]_L:L\otimes L\rightarrow L$ such that
\begin{align*}
[x,[y,z]_L]_L=[[x,y]_L,z]_L-[[x,z]_L,y]_L,~\forall~~x,y,z\in L.
\end{align*}
\end{definition}

In this paper, we only consider right Leibniz algebras.
\begin{definition}\cite{LP}
A representation of a Leibniz algebra $(L, [\cdot,\cdot]_L)$ is a triple $(V;\rho^l,\rho^r)$, where $V$ is a vector space and $\rho^l,\rho^r:L\rightarrow  gl(V)$ are linear maps such that the following equalities hold for all $x,y\in L$,
\begin{align*}
\rho^l([x,y]_L)=&[\rho^r(y),\rho^l(x)],\\
\rho^l([x,y]_L)=&\rho^l(x)\rho^l(y)+\rho^r(y)\rho^l(x),\\
\rho^r([x,y]_L)=&\rho^r(y)\rho^r(x)-\rho^r(x)\rho^r(y).
\end{align*}
\end{definition}

Let $(L, [\cdot,\cdot]_L)$ be a Leibniz algebra and $(V;\rho^l,\rho^r)$ a representation. Define a bilinear bracket $[\cdot,\cdot]_V:(L\oplus V)\otimes (L\oplus V)\rightarrow L\oplus V$ by
\begin{align*}
[x+u,y+v]_V=[x,y]_L+\rho^l(x)v+\rho^r(y)u,
\end{align*}
for all $x,y\in L$ and $u,v\in V$.

Recall the cohomology theory of Leibniz algebras, see \cite{ALO}. For a Leibniz algebra $L$ and a representation $(V;\rho^l,\rho^r)$, the cochains spaces are defined by
\begin{align*}
C^0_L(L,V)=V,~~~C^n_L(L,V)={\rm Hom} (L^{\otimes n},V),~~n>0.
\end{align*}

Let $d^n:C^n_L(L,V) \rightarrow C^{n+1}_L(L,V)$ be defined by
\begin{align*}
(d^nf)(x_1,...x_{n+1})=&\rho^l(x_1)f(x_2,...x_{n+1})+\sum\limits_{i=2}^{n+1}(-1)^i\rho^r(x_i)f(x_1,...,\widehat{x_i},...x_{n+1})\\
&+\sum\limits_{1 \leq i<j \leq n+1}^{}(-1)^{j+1}f(x_1,...x_{i-1},[x_i,x_j]_L,x_{i+1},...,\widehat{x_j},...,x_{n+1}),
\end{align*}
where $f\in C^n_L(L,V)$, $x_i\in L$ and the sign $~\widehat{}~$  indicates that the element below must be omitted.
The $n$-th cohomology group is  defined by
\begin{align*}
H^n_L(L,V)=Z^n_L(L,V)/B^n_L(L,V),
\end{align*}
where the elements $Z^n_L(L,V)=ker d^n$ and $B^n_L(L,V)= Im d ^{n-1}$ are $n$-cocycles and $n$-coboundaries, respectively.

The elements $f\in Z^1_L(L,V)$ and $g\in Z^2_L(L,V)$ are defined as follows
\begin{align*}
\rho^l(x)f(y)+\rho^r(y)f(x)-f([x,y]_L)=0,
\end{align*}
and,
\begin{align*}
\rho^l(x)g(y,z)-\rho^r(z)g(x,y)+\rho^r(y)g(x,z)+g(x,[y,z]_L)-g([x,y]_L,z)+g([x,z]_L,y)=0.
\end{align*}

\begin{definition}\cite{TSZ}
A linear map $T:V\rightarrow L$ is called a relative Rota-Baxter operator on a Leibniz algebra $(L, [\cdot,\cdot]_L)$  with respect to a representation $(V;\rho^l,\rho^r)$ if $T$ satisfies
\begin{align*}
[Tu,Tv]_L=T\big(\rho^l(Tu)v+\rho^r(Tv)u\big),
\end{align*}
for all $u,v\in V$.
\end{definition}

Now, we recall some results from \cite{TSZ}. Let $T:V\rightarrow L$ be a relative Rota-Baxter operator on a Leibniz algebra $(L, [\cdot,\cdot]_L)$ with respect to a representation $(V;\rho^l,\rho^r)$. There is a Leibniz algebra structure on $(V,[\cdot,\cdot]_T)$, where the bracket $[\cdot,\cdot]_T: V\times V\rightarrow V$ is given by
\begin{align*}
[u,v]_T=\rho^l(Tu)v+\rho^r(Tv)u,
\end{align*}
for all $u,v\in V$. There is a representation $(L;\rho^l_T,\rho^r_T)$ of $(V,[\cdot,\cdot]_T)$, where $\rho^l_T,\rho^r_T:V \rightarrow gl(L)$, is defined by
\begin{align*}
\rho^l_T(u)x=&[Tu,x]_L-T\rho^r(x)u,\\
\rho^r_T(u)x=&[x, Tu]_L-T\rho^l(x)u,
\end{align*}
for all $u\in V$ and $x\in L$.

Let $T$ be a relative Rota-Baxter operator on a Leibniz algebra $(L, [\cdot,\cdot]_L)$ with respect to a representation $(V;\rho^l,\rho^r)$. Consider the set of $n$-cochains $C^n_T(V,L)={\rm Hom}(\wedge^nV,L)$. Let $d_T^n:C^n_T(V,L) \rightarrow C^{n+1}_T(V,L)$ be the corresponding operator of the Leibniz algebra $(V,[\cdot,\cdot]_T)$ with coefficients in the representation $(L;\rho^l_T,\rho^r_T)$, define for all $f\in C^n_T(V,L)$ and $u_1,...u_{n+1}\in V$ by
\begin{align*}
(d_T^nf)(u_1,...u_{n+1})=&\rho^l_T(u_1)f(u_2,...u_{n+1})+\sum\limits_{i=2}^{n+1}(-1)^i\rho^r_T(u_i)f(u_1,...,\widehat{u_i},...u_{n+1})\\
&+\sum\limits_{1 \leq i<j \leq n+1}^{}(-1)^{j+1}f(u_1,...u_{i-1},[u_i,u_j]_T,u_{i+1},...,\widehat{u_j},...,u_{n+1}).
\end{align*}

Then the cochain complex $(C^n_T(V,L),d_T)$ is called a cochain complex of the relative Rota-Baxter operator $T$ on a Leibniz algebra $(V,[\cdot,\cdot]_T)$ with coefficients in the representation $(L;\rho^l_T,\rho^r_T)$. We denote the corresponding $n$-th cohomology group by $H^n_T(V,L)$.

In the sequel, we study the relationship between  Leibniz algebras and the corresponding  Leibniz triple systems.

\begin{proposition}
Let $(L, [\cdot,\cdot]_L)$ be a Leibniz algebra with a representation $(V;\rho^l,\rho^r)$. Define $R,M,L: \otimes^2 L\rightarrow {\rm End}(V)$ by
\begin{align*}
R(x,y)u=&\rho^r(y)\rho^r(x)(u),\\
M(x,y)u=&\rho^r(y)\rho^l(x)(u),\\
L(x,y)u=&\rho^l([x,y]_L)(u),
\end{align*}
then $(L,\{\cdot,\cdot,\cdot\}=[[\cdot,\cdot]_L,\cdot]_L)$ is a Leibniz triple system with a representation $(R,M,L)$.
\end{proposition}
\begin{proof}
Let $(L, [\cdot,\cdot]_L)$ be a Leibniz algebra with a representation $(V;\rho^l,\rho^r)$. We know that there is a Leibniz triple system structure on $L$ given by $\{\cdot,\cdot,\cdot\}=[[\cdot,\cdot]_L,\cdot]_L$. Now for all $x,y,z\in L$ and $u,v,w\in V$, we have
\begin{align*}
\{x+u,y+v,z+w\}_{L\oplus V}=&[[x+u,y+v]_{L\oplus V},z+w]_{L\oplus V}\\
                           =&[[x,y]_L+\rho^l(x)v+\rho^r(y)u, z+w]_{L\oplus V}\\
                           =&[[x,y]_L,z]_L+\rho^l([x,y]_L)w+\rho^r(z)\big(\rho^l(x)v+\rho^r(y)u\big)\\
                           =&[[x,y]_L,z]_L+L(x,y)w+M(x,z)v+R(y,z)u.
\end{align*}
Then by Eq. (\ref{2882.3}), $(R,M,L)$ is a representation of Leibniz triple system $(L,\{\cdot,\cdot,\cdot\}=[[\cdot,\cdot]_L,\cdot]_L)$.
\end{proof}

\begin{proposition}
Let $T:V\rightarrow L$ be a relative Rota-Baxter operator on a Leibniz algebra $(L,[\cdot,\cdot]_L)$ with respect to a representation $(V;\rho^l,\rho^r)$. Then $T$ is also a relative Rota-Baxter operator on the Leibniz triple system $(L,\{\cdot,\cdot,\cdot\}=[[\cdot,\cdot]_L,\cdot]_L)$ with respect to  a representation $(R,M,L)$.
\end{proposition}
\begin{proof}
For any $u,v,w\in V$, we have
\begin{align*}
\{Tu,Tv,Tw\}=&[[Tu,Tv]_L,Tw]_L=[T(\rho^l(Tu)v+\rho^r(Tv)u),Tw]_L\\
            =&[T\rho^l(Tu)v,Tw]_L+[T\rho^r(Tv)u,Tw]_L\\
            =&T\Big(\rho^l\big(T\rho^l(Tu)v\big)w+\rho^r(Tw)\rho^l(Tu)v\Big)
            +T\Big(\rho^l\big(T\rho^r(Tv)u\big)w+\rho^r(Tw)\rho^r(Tv)u\Big)\\
            =&T\Big(\rho^l([Tu,Tv]_L)w+\rho^r(Tw)\rho^l(Tu)v+\rho^r(Tw)\rho^r(Tv)u\Big)\\
            =&T\Big(L(Tu,Tv)w+M(Tu,Tw)v+R(Tv,Tw)u\Big).
\end{align*}
This implies that $T$ is a relative Rota-Baxter operator.
\end{proof}

Now there are two methods of constructing a Leibniz triple system structure on $V$ from a Leibniz algebra $(L,[\cdot,\cdot]_L)$ with respect to a representation $(V;\rho^l,\rho^r)$. On the one hand, we induce a Leibniz triple system structure $(V,\{\cdot,\cdot,\cdot\}_T)$ of the relative Rota-Baxter operator $T$ on the corresponding Leibniz triple system $(L,\{\cdot,\cdot,\cdot\}=[[\cdot,\cdot]_L,\cdot]_L)$ with respect to  a representation $(R,M,L)$, where
\begin{align*}
\{u,v,w\}_T=L(Tu,Tv)w+M(Tu,Tw)v+R(Tv,Tw)u.
\end{align*}
On the other hand, we firstly induce a Leibniz algebra $(V,[\cdot,\cdot]_T)$ of $T$ on the Leibniz algebra $(L,[\cdot,\cdot]_L)$ with respect to a representation $(V;\rho^l,\rho^r)$, and then we give a Leibniz triple system structure induced from the Leibniz algebra $(V,[\cdot,\cdot]_T)$ by the bracket $\{\cdot,\cdot,\cdot\}_T=[[\cdot,\cdot]_T,\cdot]_T$. These two methods give us the same Leibniz triple system structure on $V$.

For convince, we consider the representation related to the relative Rota-Baxter operator $T$.

\begin{lemma}
Let $(V, [\cdot,\cdot]_T)$ be a Leibniz algebra with a representation $(L;\rho^l_T,\rho^r_T)$. Define $R_T,M_T,L_T: \otimes^2 V\rightarrow {\rm End}(L)$ by
\begin{align*}
R_T(u,v)x=&\rho^r_T(v)\rho^r_T(u)(x),\\
M_T(u,v)x=&\rho^r_T(v)\rho^l_T(u)(x),\\
L_T(u,v)x=&\rho^l_T([u,v]_T)(x),
\end{align*}
then $(V,\{\cdot,\cdot,\cdot\}_T=[[\cdot,\cdot]_T,\cdot]_T)$ is a Leibniz triple system with a representation $(R_T, M_T,L_T)$.
\end{lemma}
%\begin{proof}
%Let $(V, [\cdot,\cdot]_T)$ be a Leibniz algebra with a representation $(L;\rho^l_T,\rho^r_T)$. We know that there is a Leibniz triple system structure on $V$ given by $\{\cdot,\cdot,\cdot\}_T=[[\cdot,\cdot]_T,\cdot]_T$. Now for all $u,v,w\in V$ and $x,y,z\in L$ , we have
%\begin{align*}
%\{u+x,v+y,w+z\}_{V\oplus L}=&[[u+x,v+y]_{V\oplus L},w+z]_{V\oplus L}\\
                           %=&[[u,v]_T+\rho^l_T(u)y+\rho^r_T(v)x, w+z]_{V\oplus L}\\
                           %=&[[u,v]_T,w]_T+\rho^l_T([u,v]_L)z+\rho^r_T(w)\big(\rho^l_T(u)y+\rho^r_T(v)x\big)\\
                           %=&[[u,v]_T,w]_T+L_T(u,v)z+M_T(u,w)y+R_T(v,w)x.
%\end{align*}
%Then by Eq. (\ref{2882.3}), $(R_T,M_T,L_T)$ is a representation of Leibniz triple system $(V,\{\cdot,\cdot,\cdot\}_T=[[\cdot,\cdot]_T,\cdot]_T)$.
%\end{proof}

\begin{theorem}
Every 1-cocycle for the cohomology of a Leibniz algebra $(V, [\cdot,\cdot]_T)$ with respect to  a representation $(L;\rho^l_T,\rho^r_T)$ is a 1-cocycle for the cohomology of a Leibniz triple system $(V,\{\cdot,\cdot,\cdot\}_T=[[\cdot,\cdot]_T,\cdot]_T)$ with respect to  a representation $(R_T,M_T,L_T)$.
\end{theorem}
\begin{proof}
Let $\phi$ be a 1-cocycle of the cohomology of the Leibniz algebra $(V, [\cdot,\cdot]_T)$ with respect to  a representation $(L;\rho^l_T,\rho^r_T)$, then for all $u,v\in V$,
\begin{align*}
d^1_T\phi(u,v)=\rho^l_T(u)\phi(v)+\rho^r_T(v)\phi(u)-\phi([u,v]_T)=0.
\end{align*}
For all $u,v,w\in V$,
\begin{align*}
&\partial^1\phi(u,v,w)\\
=&L_T(u,v)\phi(w)+M_T(u,w)\phi(v)+R_T(v,w)\phi(u)-\phi(\{u,v,w\}_T)\\
=&\rho^l_T([u,v]_T)\phi(w)+\rho^r_T(w)\rho^l_T(u)\phi(v)+\rho^r_T(v)\rho^r_T(w)\phi(u)-\phi([[u,v]_T,w]_T)\\
=&\rho^l_T([u,v]_T)\phi(w)+\rho^r_T(w)\rho^l_T(u)\phi(v)+\rho^r_T(v)\rho^r_T(w)\phi(u)
-\rho^r_T(w)\phi([u,v]_T)\rho^l_T([u,v]_T)\phi(w)\\
=&\rho^r_T(w)(\rho^l_T(u)\phi(v)+\rho^r_T(v)\phi(u)-\phi([u,v]_T))\\
=&0,
\end{align*}
which implies that $\phi$ is a 1-cocycle for the cohomology of a Leibniz triple system $(V,\{\cdot,\cdot,\cdot\}_T=[[\cdot,\cdot]_T,\cdot]_T)$ with respect to  a representation $(R_T,M_T,L_T)$.
\end{proof}

\begin{theorem}
Let $\phi\in Z^2_T(V,L)$. Then $\omega(u,v,w)=\phi([u,v]_T,w)+\rho^r_T(w)\phi(u,v)$ is a $3$-cocycle of the Leibniz triple system $(V,\{\cdot,\cdot,\cdot\}_T=[[\cdot,\cdot]_T,\cdot]_T)$ with respect to  the representation $(R_T,M_T,L_T)$.
\end{theorem}
\begin{proof}
Let $\phi\in Z^2_T(V,L)$, for any $u_1,u_2,u_3,u_4,u_5\in V$, we have
\begin{align*}
&\partial^3_1(\omega)(u_1,u_2,u_3,u_4,u_5)\\
=&\omega(u_1,u_2,\{u_3,u_4,u_5\}_T)-\omega(\{u_1,u_2,u_3\}_T,u_4,u_5)+\omega(\{u_1,u_2,u_4\}_T,u_3,u_5)\\
&+\omega(\{u_1,u_2,u_5\}_T,u_3,u_4)-\omega(\{u_1,u_2,u_5\}_T,u_4,u_3)+L_T(u_1,u_2)\omega(u_3,u_4,u_5)\\
&-R_T(u_4,u_5)\omega(u_1,u_2,u_3)+R_T(u_3,u_5)\omega(u_1,u_2,u_4)+R_T(u_3,u_4)\omega(u_1,u_2,u_5)\\
&-R_T(u_4,u_3)\omega(u_1,u_2,u_5)\\
=&\phi([u_1,u_2]_T,[[u_3,u_4]_T,u_5]_T)-\phi([[[u_1,u_2]_T,u_3]_T,u_4]_T,u_5)+\phi([[[u_1,u_2]_T,u_4]_T,u_3]_T,u_5)\\
&+\phi([[[u_1,u_2]_T,u_5]_T,u_3]_T,u_4)-\phi([[[u_1,u_2]_T,u_5]_T,u_4]_T,u_3)+\rho^l_T(u_1)\rho^l_T(u_2)\phi([u_3,u_4]_T,u_5)\\
&+\rho^r_T(u_2)\rho^l_T(u_1)\phi([u_3,u_4]_T,u_5)-\rho^r_T(u_5)\rho^r_T(u_4)\phi([u_1,u_2]_T,u_3)+\rho^r_T(u_5)\rho^r_T(u_3)\phi([u_1,u_2]_T,u_4)\\
&+\rho^r_T(u_4)\rho^r_T(u_3)\phi([u_1,u_2]_T,u_5)-\rho^r_T(u_3)\rho^r_T(u_4)\phi([u_1,u_2]_T,u_5)+\rho^r_T([[u_3,u_4]_T,u_5]_T)\phi(u_1,u_2)\\
&+\rho^r_T(u_5)\phi([[u_1,u_2]_T,u_3]_T,u_4)+\rho^r_T(u_5)\phi([[u_1,u_2]_T,u_4]_T,u_3)+\rho^r_T(u_4)\phi([[u_1,u_2]_T,u_5]_T,u_3)\\
&-\rho^r_T(u_3)\phi([[u_1,u_2]_T,u_5]_T,u_4)+\rho^l_T(u_1)\rho^l_T(u_2)\rho^r_T(u_5)\phi(u_3,u_4)
+\rho^r_T(u_2)\rho^l_T(u_1)\rho^r_T(u_5)\phi(u_3,u_4)\\
&-\rho^r_T(u_5)\rho^r_T(u_4)\rho^r_T(u_3)\phi(u_1,u_2)+\rho^r_T(u_5)\rho^r_T(u_3)\rho^r_T(u_4)\phi(u_1,u_2)
+\rho^r_T(u_4)\rho^r_T(u_3)\rho^r_T(u_5)\phi(u_1,u_2)\\
&-\rho^r_T(u_3)\rho^r_T(u_4)\rho^r_T(u_5)\phi(u_1,u_2)\\
=&\phi([u_1,u_2]_T,[[u_3,u_4]_T,u_5]_T)-\phi([[u_1,u_2]_T,[u_3,u_4]_T]_T,u_5)+\phi([[[u_1,u_2]_T,u_5]_T,u_3]_T,u_4)\\
&-\phi([[[u_1,u_2]_T,u_5]_T,u_4]_T,u_3)+\rho^l_T([u_1,u_2]_T)\phi([u_3,u_4]_T,u_5)-\rho^r_T(u_5)\rho^r_T(u_4)\phi([u_1,u_2]_T,u_3)\\
&+\rho^r_T(u_5)\rho^r_T(u_3)\phi([u_1,u_2]_T,u_4)+\rho^r_T([u_3,u_4]_T)\phi([u_1,u_2]_T,u_5)+\rho^r_T([[u_3,u_4]_T,u_5]_T)\phi(u_1,u_2)\\
&+\rho^r_T(u_5)\phi([[u_1,u_2]_T,u_3]_T,u_4)+\rho^r_T(u_5)\phi([[u_1,u_2]_T,u_4]_T,u_3)+\rho^r_T(u_4)\phi([[u_1,u_2]_T,u_5]_T,u_3)\\
&-\rho^r_T(u_3)\phi([[u_1,u_2]_T,u_5]_T,u_4)+\rho^l_T([u_1,u_2]_T)\rho^r_T(u_5)\phi(u_3,u_4)-\rho^r_T(u_5)\rho^r_T([u_3,u_4]_T)\phi(u_1,u_2)\\
&+\rho^r_T([u_3,u_4]_T)\rho^r_T(u_5)\phi(u_1,u_2)\\
=&\phi([u_1,u_2]_T,[[u_3,u_4]_T,u_5]_T)-\phi([[u_1,u_2]_T,[u_3,u_4]_T]_T,u_5)+\phi([[u_1,u_2]_T,u_5]_T,[u_3,u_4]_T)\\
&+\rho^l_T([[u_1,u_2]_T,u_5]_T)\phi(u_3,u_4)+\rho^l_T([u_1,u_2]_T)\phi([u_3,u_4]_T,u_5)+\rho^r_T([u_3,u_4]_T)\phi([u_1,u_2]_T,u_5)\\
&+\rho^l_T([u_1,u_2]_T)\rho^r_T(u_5)\phi(u_3,u_4)-\rho^r_T(u_5)\rho^l_T([u_1,u_2]_T)\phi(u_3,u_4)-\rho^r_T(u_5)\phi([u_1,u_2]_T,[u_3,u_4]_T)\\
=&\phi([u_1,u_2]_T,[[u_3,u_4]_T,u_5]_T)-\phi([[u_1,u_2]_T,[u_3,u_4]_T]_T,u_5)+\phi([[u_1,u_2]_T,u_5]_T,[u_3,u_4]_T)\\
&+\rho^l_T([u_1,u_2]_T)\phi([u_3,u_4]_T,u_5)+\rho^r_T([u_3,u_4]_T)\phi([u_1,u_2]_T,u_5)-\rho^r_T(u_5)\phi([u_1,u_2]_T,[u_3,u_4]_T)\\
=&0,
\end{align*}
and
\begin{align*}
&\partial^3_2(\omega)(u_1,u_2,u_3,u_4,u_5)\\
=&\omega(u_1,\{u_2,u_3,u_4\}_T,u_5)-\omega(\{u_1,u_2,u_3\}_T,u_4,u_5)+\omega(\{u_1,u_3,u_2\}_T,u_4,u_5)\\
&+\omega(\{u_1,u_4,u_2\}_T,u_3,u_5)-\omega(\{u_1,u_4,u_3\}_T,u_2,u_5)+M_T(u_1,u_5)\omega(u_2,u_3,u_4)\\
&-R_T(u_4,u_5)\omega(u_1,u_2,u_3)+R_T(u_4,u_5)\omega(u_1,u_3,u_2)+R_T(u_3,u_5)\omega(u_1,u_4,u_2)\\
&-R_T(u_2,u_5)\omega(u_1,u_4,u_3)\\
=&\phi([u_1,[[u_2,u_3]_T,u_4]_T]_T,u_5)+\rho^r_T(u_5)\phi(u_1,[[u_2,u_3]_T,u_4]_T)-\phi([[[u_1,u_2]_T,u_3]_T,u_4]_T,u_5)\\
&-\rho^r_T(u_5)\phi([[u_1,u_2]_T,u_3]_T,u_4)+\phi([[[u_1,u_3]_T,u_2]_T,u_4]_T,u_5)+\rho^r_T(u_5)\phi([[u_1,u_3]_T,u_2]_T,u_4)\\
&+\phi([[[u_1,u_4]_T,u_2]_T,u_3]_T,u_5)+\rho^r_T(u_5)\phi([[u_1,u_4]_T,u_2]_T,u_3)-\phi([[[u_1,u_4]_T,u_3]_T,u_2]_T,u_5)\\
&-\rho^r_T(u_5)\phi([[u_1,u_4]_T,u_3]_T,u_2)+\rho^r_T(u_5)\rho^l_T(u_1)\phi([u_2,u_3]_T,u_4)+\rho^r_T(u_5)\rho^l_T(u_1)\rho^r_T(u_4)\phi(u_2,u_3)\\
&-\rho^r_T(u_5)\rho^r_T(u_4)\phi([u_1,u_2]_T,u_3)-\rho^r_T(u_5)\rho^r_T(u_4)\rho^r_T(u_3)\phi(u_1,u_2)
+\rho^r_T(u_5)\rho^r_T(u_4)\phi([u_1,u_3]_T,u_2)\\
&+\rho^r_T(u_5)\rho^r_T(u_4)\rho^r_T(u_2)\phi(u_1,u_3)+\rho^r_T(u_5)\rho^r_T(u_3)\phi([u_1,u_4]_T,u_2)
+\rho^r_T(u_5)\rho^r_T(u_3)\rho^r_T(u_2)\phi(u_1,u_4)\\
&-\rho^r_T(u_5)\rho^r_T(u_2)\phi([u_1,u_4]_T,u_3)-\rho^r_T(u_5)\rho^r_T(u_2)\rho^r_T(u_3)\phi(u_1,u_4)\\
=&\rho^r_T(u_5)\phi(u_1,[[u_2,u_3]_T,u_4]_T)-\rho^r_T(u_5)\phi([[u_1,u_2]_T,u_3]_T,u_4)+\rho^r_T(u_5)\phi([[u_1,u_3]_T,u_2]_T,u_4)\\
&+\rho^r_T(u_5)\phi([[u_1,u_4]_T,u_2]_T,u_3)-\rho^r_T(u_5)\phi([[u_1,u_4]_T,u_3]_T,u_2)+\rho^r_T(u_5)\rho^l_T(u_1)\phi([u_2,u_3]_T,u_4)\\
&+\rho^r_T(u_5)\rho^l_T(u_1)\rho^r_T(u_4)\phi(u_2,u_3)-\rho^r_T(u_5)\rho^r_T(u_4)\phi([u_1,u_2]_T,u_3)
-\rho^r_T(u_5)\rho^r_T(u_4)\rho^r_T(u_3)\phi(u_1,u_2)\\
&+\rho^r_T(u_5)\rho^r_T(u_4)\phi([u_1,u_3]_T,u_2)+\rho^r_T(u_5)\rho^r_T(u_4)\rho^r_T(u_2)\phi(u_1,u_3)
+\rho^r_T(u_5)\rho^r_T(u_3)\phi([u_1,u_4]_T,u_2)\\
&+\rho^r_T(u_5)\rho^r_T(u_3)\rho^r_T(u_2)\phi(u_1,u_4)-\rho^r_T(u_5)\rho^r_T(u_2)\phi([u_1,u_4]_T,u_3)
-\rho^r_T(u_5)\rho^r_T(u_2)\rho^r_T(u_3)\phi(u_1,u_4)\\
=&\rho^r_T(u_5)\phi(u_1,[[u_2,u_3]_T,u_4]_T)-\rho^r_T(u_5)\phi([u_1,[u_2,u_3]_T]_T,u_4)+\rho^r_T(u_5)\phi([[u_1,u_4]_T,u_2]_T,u_3)\\
&-\rho^r_T(u_5)\phi([[u_1,u_4]_T,u_3]_T,u_2)+\rho^r_T(u_5)\rho^l_T(u_1)\phi([u_2,u_3]_T,u_4)+\rho^r_T(u_5)\rho^l_T(u_1)\rho^r_T(u_4)\phi(u_2,u_3)\\
&-\rho^r_T(u_5)\rho^r_T(u_4)\rho^l_T(u_1)\phi(u_2,u_3)-\rho^r_T(u_5)\rho^r_T(u_4)\phi(u_1,[u_2,u_3]_T)
+\rho^r_T(u_5)\rho^r_T(u_3)\phi([u_1,u_4]_T,u_2)\\
&+\rho^r_T(u_5)\rho^r_T([u_2,u_3]_T)\phi(u_1,u_4)-\rho^r_T(u_5)\rho^r_T(u_2)\phi([u_1,u_4]_T,u_3)\\
=&\rho^r_T(u_5)\phi([[u_1,u_4]_T,u_2]_T,u_3)-\rho^r_T(u_5)\phi([[u_1,u_4]_T,u_3]_T,u_2)-\rho^r_T(u_5)\rho^l_T([u_1,u_4]_T)\phi(u_2,u_3)\\
&+\rho^r_T(u_5)\rho^r_T(u_3)\phi([u_1,u_4]_T,u_2)-\rho^r_T(u_5)\rho^r_T(u_2)\phi([u_1,u_4]_T,u_3)-\rho^r_T(u_5)\phi([u_1,u_4]_T,[u_2,u_3]_T)\\
=&\rho^r_T(u_5)\Big(\phi([[u_1,u_4]_T,u_2]_T,u_3)-\phi([[u_1,u_4]_T,u_3]_T,u_2)-\rho^l_T([u_1,u_4]_T)\phi(u_2,u_3)\\
&+\rho^r_T(u_3)\phi([u_1,u_4]_T,u_2)-\rho^r_T(u_2)\phi([u_1,u_4]_T,u_3)-\phi([u_1,u_4]_T,[u_2,u_3]_T)\Big)\\
=&0.
\end{align*}
Since $\phi\in Z^2_T(V,L)$, we get $\partial^3_1(\omega)(u_1,u_2,u_3,u_4,u_5)=\partial^3_2(\omega)(u_1,u_2,u_3,u_4,u_5)=0$.
\end{proof}

\begin{lemma}\label{lem285,7}
Let $\alpha\in C^1_T(V,L)$. Then
\begin{align*}
\partial^1(\alpha)(u,v,w)=d_T(\alpha)([u,v]_T,w)+\rho^r_T(w)d(\alpha)(u,v).
\end{align*}
\end{lemma}
\begin{proof}
For any $u,v,w\in V$, we have
\begin{align*}
&\partial^1(\alpha)(u,v,w)\\
=&L_T(u,v)\alpha(w)+M_T(u,w)\alpha(v)+R_T(v,w)\alpha(u)-\alpha(\{u,v,w\}_T)\\
=&\rho^l_T([u,v]_T)\alpha(w)+\rho^r_T(w)\rho^l_T(u)\alpha(v)+\rho^r_T(w)\rho^r_T(v)\alpha(u)-\alpha([[u,v]_T,w]_T)\\
=&\rho^l_T([u,v]_T)\alpha(w)+\rho^r_T(w)\alpha([u,v]_T)-\alpha([[u,v]_T,w]_T)\\
&+\rho^r_T(w)\Big(\rho^l_T(u)\alpha(v)+\rho^r_T(v)\alpha(u)-\alpha([u,v]_T)\Big)\\
=&d(\alpha)([u,v]_T,w)+\rho^r_T(w)d(\alpha)(u,v).
\end{align*}
\end{proof}

\begin{proposition}
Let $\phi_1,\phi_2\in Z^2_T(V,L)$. If $\phi_1,\phi_2$ are in the same cohomology class then $\omega_1,\omega_2$ defined by:
\begin{align*}
\omega_i(u,v,w)=\phi_i([u,v]_T,w)+\rho^r_T(w)\phi_i(u,v),~~i=1,2,
\end{align*}
are in the same cohomology class of the associated Leibniz triple system.
\end{proposition}
\begin{proof}
Let $\phi_1,\phi_2\in Z^2_T(V,L)$ be two cocycles in the same cohomology class, that is
\begin{align*}
\phi_2-\phi_1=d(\alpha),~~\alpha\in C^1_T(V,L).
\end{align*}
According to Lemma \ref{lem285,7}, we have
\begin{align*}
&\omega_2(u,v,w)-\omega_1(u,v,w)\\
=&(\phi_2-\phi_1)([u,v]_T,w)+\rho^r_T(w)(\phi_2-\phi_1)(u,v)\\
=&d(\alpha)([u,v]_T,w)+\rho^r_T(w)d(\alpha)(u,v)
=\delta^1(\alpha)(u,v,w),
\end{align*}
$\alpha\in C^1_T(V,L)$, which means that $\omega_1$ and $\omega_2$ are in the same cohomology class.
\end{proof}

\end{document}